\documentclass[a4paper, 12pt]{article}
\usepackage{amssymb}
\usepackage{bm}
\usepackage{textcomp}
\usepackage{mathrsfs}
\usepackage{amsmath}
\usepackage{amsfonts}
\usepackage[T1]{fontenc}
\usepackage[utf8]{inputenc}
\usepackage{amsthm}
\usepackage{graphicx}
\usepackage{amsmath}

\usepackage{amsthm}
\usepackage{array}
\usepackage{cases}
\usepackage{enumerate}
\usepackage{clrscode}
\usepackage{listings}
\usepackage[ruled,vlined,english,linesnumbered]{algorithm2e}
\usepackage{url}

\usepackage{thmtools}
\usepackage{thm-restate}

\usepackage{enumerate,enumitem,todonotes}
\makeatletter
\def\th@plain{%
  \itshape % body font
}
\makeatother

\makeatletter
\renewenvironment{proof}[1][\proofname]{\par
  \pushQED{\qed}%
  \normalfont \topsep6\p@\@plus6\p@\relax
  \trivlist
  \item[\hskip\labelsep
        \bfseries
    #1\@addpunct{.}]\ignorespaces
}{%
  \popQED\endtrivlist\@endpefalse
}
\makeatother

\newtheorem{theorem}{Theorem}[section]

\numberwithin{equation}{section}

\newtheorem{lemma}[theorem]{Lemma}
\newtheorem{thm}{Theorem}[section]
\newtheorem{cor}[thm]{Corollary}

\numberwithin{equation}{section}

\usepackage[varg]{txfonts}
\usepackage{graphicx, epsfig, subfigure}
\usepackage{enumerate} %% žÄ±äÁÐ±í±êºÅÑùÊœºê°ü Æäºó¿ÉœÓÑ¡Ïî[a,A,i,I,1]
\usepackage[square, numbers, sort&compress]{natbib} %% Ö§³ÖÒýÓÃµÄºê°ü
\ifx\pdfoutput\undefined
 \usepackage[dvipdfm,%
  pdfstartview=FitH, %% FitH±íÊŸŽò¿ªpdfÎÄŒþµÄÊ±ºò×Ô¶¯Ëõ·ÅÒ³ÃæŽóÐ¡µœÊÓŽ°ŽóÐ¡£¬Fits±íÊŸÒ³¿íÊÊÓŠwindows ŽóÐ¡
  bookmarks=true,%
  bookmarksnumbered=true, %% °ÑÕÂœÚµÄ±êºÅ·ÅÈëÊéÇ©ÖÐ£¬È±Ê¡Îªflase
  bookmarksopen=true, %% Žò¿ªÊéÇ©Ê÷
  plainpages=false,%
  pdfpagelabels,%
  colorlinks=true, %% ²ÊÉ«ÎÄ×ÖÁŽœÓŽò¿ª, È¥µô·œ¿ò
  linkcolor=blue, %% Ä¿ÂŒÁŽœÓµÄÑÕÉ«ÎªÀ¶É«
  citecolor=blue,%
  urlcolor=black,
  pdfborder=001]{hyperref}
  \AtBeginDvi{}  %  GBK  ->  Unicode
\else
 \usepackage[pdftex,%
  pdfstartview=FitH, %% FitH±íÊŸŽò¿ªpdfÎÄŒþµÄÊ±ºò×Ô¶¯Ëõ·ÅÒ³ÃæŽóÐ¡µœÊÓŽ°ŽóÐ¡£¬Fits±íÊŸÒ³¿íÊÊÓŠwindows ŽóÐ¡
  bookmarks=true,%
  bookmarksnumbered=true, %% °ÑÕÂœÚµÄ±êºÅ·ÅÈëÊéÇ©ÖÐ£¬È±Ê¡Îªflase
  bookmarksopen=true, %% Žò¿ªÊéÇ©Ê÷
  plainpages=false,%
  pdfpagelabels,%
  colorlinks=true, %% ²ÊÉ«ÎÄ×ÖÁŽœÓŽò¿ª
  linkcolor=blue, %% Ä¿ÂŒÁŽœÓµÄÑÕÉ«ÎªÀ¶É«
  citecolor=blue,%
  urlcolor=black,
  pdfborder=001]{hyperref}%%% ²»ÄÜºÍciteºê°üÍ¬Ê±Ê¹ÓÃ
\fi

\numberwithin{equation}{section}

\begin{document}

\title{\LARGE Exact rainbow numbers of cycle-related graphs in multi-hubbed wheels}
\author{Mengyao Dai~~~~~~~~ Xin Zhang\thanks{Corresponding author.}\\
{\small School of Mathematics and Statistics, Xidian University, Xi'an, 710071, China}\\
{\small  mengyaodai@stu.xidian.edu.cn \quad  xzhang@xidian.edu.cn}
}

%\date{}

\maketitle

\begin{abstract}\baselineskip 0.60cm

The rainbow number ${\rm rb}(G, H)$ is the minimum number of colors $k$ for which any edge-coloring of $G$ with at least $k$ colors guarantees a rainbow subgraph isomorphic to $H$. The rainbow number has many applications in diverse fields such as wireless communication networks, cryptography, bioinformatics, and social network analysis. In this paper, we determine the exact rainbow number $\mathrm{rb}(G, H)$ where $G$ is a multi-hubbed wheel graph $W_d(s)$, defined as the join of $s$ isolated vertices and a cycle $C_d$ of length $d$ (i.e., $W_d(s) = \overline{K_s} + C_d$), and $H = \theta_{t,\ell}$ represents a cycle $C_t$ of length $t$ with $0 \leq \ell \leq t-3$ chords emanating from a common vertex, by establishing
\[
{\rm rb}(W_{d}(s), \theta_{t,\ell}) = 
\begin{cases}
\left\lfloor \dfrac{2t - 5}{t - 2}d \right\rfloor + 1, & \text{if } \ell=t-3,~s = 1 \text{ and } t\ge 4, \\[10pt]
\left\lfloor \dfrac{3t-10}{t - 3}d \right\rfloor + 1, & \text{if } \ell=t-3,~s = 2\text{ and } t\ge 6,\\[10pt]
\left\lfloor \dfrac{(s + 1)t - (3s + 4)}{t - 3}d \right\rfloor + 1, & \text{if } \ell=t-3,~s \geq 3\text{ and } t\ge 7,\\[10pt]
\left\lfloor \dfrac{2t - 7}{t - 3}d \right\rfloor + 1, & \text{if } s = 1 \text{ and } t\ge \max\{5,\ell+4\},
\end{cases}
\]
when $d\geq 3t-5$, with all bounds for the parameter $t$ presented here being tight. This addresses the problems proposed by Jakhar, Budden, and Moun (2025), which involve investigating the rainbow numbers of large cycles and large chorded cycles in wheel graphs (specifically corresponding to the cases in our framework where $s=1$ and $\ell\in \{0,1\}$). Furthermore, it completely determines the rainbow numbers of cycles of arbitrary length in large wheel graphs, thereby generalizing a result of Lan, Shi, and Song (2019).

% ---------------------------------

% Given a family of graphs $G$ and a subgraph $G$ that appears in at least one of the graphs in $G$. The rainbow number for $H$ with respect to $G$, denoted ${\rm {\rm rb}}(G,H)$ , is the minimum number $k$ of colors for which any edge coloring of $G$ using at least $k$ colors ensures the existence of a rainbow subgraph that is isomorphic to $H$ . That is, any two edges of $H$ are assigned different colors. In 1973, Erd{\H o}s, P. and Simonovits, M. and S{\'o}s, V. T introduced the anti-Ramsey number $f(n, H)$ as the maximum integer $t$ for which there exists an edge-coloring of $K_n$ using exactly $t$ colors that contains no rainbow subgraph isomorphic to $H$. It follows directly from this definition that ${\rm {\rm rb}}(K_n,H) = f(n,H) + 1$. 
% In this paper, denote by $W_{d}$ the wheel graph of order $d+1$ and $W^*_d$ the double-hubbed wheel graph of order $d + 2$. We determine the exact rainbow numbers for the subgraphs $F_t$ and $\theta_{t,n}$ with respect to the host graphs $W_{d}$ and $W^*_d$, where $F_t$ denotes the $t-cycle$ $C_t$ with full chords. Let $\theta_{t,n}$ denote the $t-cycle$ $C_t$ with $n$ chords, where $1\le n\le t-4$. When $n=t-3$, Graph $F_t$ and Graph $\theta_{t,n}$ represent the same situation, that is, the $t-cycle$ with  full chords.

\vspace{3mm}\noindent \emph{Keywords: rainbow number; anti-Ramsey number; edge coloring}.

\end{abstract}

\baselineskip 0.60cm

\section{Introduction}\label{sec1}

A subgraph $H$ of an edge-colored graph $G$ is called an \textit{rainbow subgraph} if no two edges of $H$ share the same color. In 1973, Erd{\H o}s, Simonovits, and S\'os \cite{MR379258} introduced the \textit{anti-Ramsey number} $ar_{\mathcal{P}}(n, H)$ as the maximum integer $t$ such that there exists an edge-coloring of $K_n$ using exactly $t$ colors that contains no rainbow subgraph isomorphic to $H$. This concept is closely related to the \textit{rainbow number} ${\rm rb}(G, H)$, defined as the minimum number of colors $k$ for which any edge-coloring of $G$ with at least $k$ colors guarantees a rainbow subgraph isomorphic to $H$. It follows directly from these definitions that  
${\rm rb}(K_n, H) = ar_{\mathcal{P}}(n, H) + 1.$

In wireless sensor networks or satellite communication, signal transmission may fail due to frequency conflicts or malicious interference. By assigning distinct frequency bands (colors) to edges, the rainbow number $\mathrm{rb}(G, H)$ quantifies the minimum number of bands required to guarantee an interference-resistant path (rainbow path) between any two nodes. For instance, if $\mathrm{rb}(G, P_3) = k$, then at least $k$ bands are needed to ensure a path without consecutive same-frequency edges ($P_3$ denotes a 2-edge path). In encrypted networks, rainbow paths model multi-hop encrypted relays, where each edge represents an encryption layer and colors denote unique keys; the minimality of $\mathrm{rb}(G, H)$ prevents attackers from decrypting entire paths by intercepting monochromatic edges, enhancing security. In protein interaction networks, vertices represent proteins, edges denote interactions, and colors signify functional categories (e.g., catalysis, signaling); a small $\mathrm{rb}(G, K_3)$ indicates stable triads of functionally diverse proteins (e.g., kinase-phosphatase-substrate complexes), critical for understanding cellular regulation. In social networks, colors reflect topics (e.g., tech, entertainment, politics), and $\mathrm{rb}(G, S_4)$ (for star subgraphs) measures the difficulty of accessing diverse information via colorful paths; a high $\mathrm{rb}(G, S_4)$ suggests topic homogeneity among an influencer's followers, necessitating algorithmic interventions to diversify content and mitigate ``filter bubbles.''

Rainbow numbers (or anti-Ramsey numbers) for some special graph classes have been studied in the literature \cite{Alon1983OnAC,Jiang_West_2003,CHEN20093370,DBLP:journals/combinatorics/FujitaKSS09,10.1016/j.disc.2011.10.017,JENDROL2014158,XU2016193}; researchers have also investigated rainbow numbers in specific host graph families, including 
Kneser graphs \cite{JIN2020124724}, 
bipartite graphs \cite{LI20092575}, 
complete multipartite graphs \cite{10.1007/s00373-021-02302-z,ZHANG2022113100,AN2026114692},
planar graphs \cite{GYORI2025114523}, 
bipartite planar graph \cite{JIN2022127356,REN202437,10.1016/j.disc.2025.114596},
plane triangulations \cite{JENDROL2014158,10.1002/jgt.21803,CHEN20192106,LAN20193216,QIN2019221,QIN2020124888,QIN2021125918,QIN2021112301}, and
outerplanar graphs \cite{JIN20182846,sym14061252,10.1016/j.dam.2023.11.049}. Notably, the hypergraph variant of this problem has received dedicated attention
\cite{zkahya2013AntiRamseyNO,doi:10.1137/19M1244950,JIN2021112594,XUE2022112782,Liu_2023}.

In particular, Jendrol', Schiermeyer, and Tu \cite{JENDROL2014158} pioneered the study of rainbow numbers in plane triangulations by focusing on matchings; for a matching of size $k$ (denoted $M_k$), they established exact value for ${\rm rb}(T_n,M_t)$ when $t \leq 4$, and derived lower/upper bounds for all $t \geq 5$ with $n \geq 2t$, where $T_n$ denotes the class of plane triangulations on $n$ vertices and ${\rm rb}(T_n,M_t):= \max\{{\rm rb}(G,M_t)~|~G\in T_n\}$.
Subsequently, Qin, Lan, and Shi \cite{QIN2019221} determined the exact value of ${\rm rb}(T_n,M_5)$ for $n \geq 11$ and improved the upper bound to
$\mathrm{rb}(T_n,M_t) \leq 2n + 6t - 16$ for all $t \geq 5$ and $n \geq 2t$.
Chen, Lan, and Song \cite{CHEN20192106} further obtained the exact value of $\mathrm{rb}(T_n,M_6)$ and refined lower/upper bounds for general $\mathrm{rb}(T_n,M_t)$. For sufficiently large $n$ relative to $t \geq 7$, Qin, Lan, Shi, and Yue \cite{QIN2021112301} established a strengthened lower bound for $\mathrm{rb}(T_n,M_t)$.

In 2015, Hor{\v n}{\'a}k, Jendrol', Schiermeyer, and Sot{\'a}k \cite{10.1002/jgt.21803} initiated the study of rainbow numbers for cycles in planar graphs by establishing that $\mathrm{rb}(T_n, C_3) = \left\lfloor \frac{3n}{2} \right\rfloor - 2$ holds for all $n \geq 4$. Regarding the general case of $\mathrm{rb}(T_n, C_k)$, they established lower bounds for all integers $k\ge 4$, and additionally derived upper bounds specifically for $k=4$ and $k=5$; subsequently, Lan, Shi, and Song \cite{LAN20193216} improved the lower bound for $k \geq 5$ with $n \geq k^2 - k$, and determined explicit upper bounds for 
$\mathrm{rb}(T_n, C_6)$ and $\mathrm{rb}(T_n, C_7)$.

The study of $\mathrm{rb}(T_n, H)$ is closely related to the investigation of $\mathrm{rb}(W_d, H)$, where $W_d$ denotes a wheel graph on $d+1$ vertices; this connection arises from the structural property that any vertex with degree $d \geq 3$ in a plane triangulation induces a $W_d$ subgraph via its closed neighborhood. Breakthroughs by Qin, Lei, and Li \cite{QIN2020124888} have substantially advanced this relationship by determining exact values for rainbow numbers in wheel graphs: they proved $\mathrm{rb}(W_d, C_3) = d + 2$ for triangles, and extended these results to pendant-augmented structures, obtaining $\mathrm{rb}(W_d, H) = d + 2$ when $H$ is a triangle with a single pendant edge, $\mathrm{rb}(W_d, H) = d + 4$ when $H$ is a triangle with two concentrated pendant edges, and $\mathrm{rb}(W_d, H) = \left\lfloor \frac{4d}{3} \right\rfloor + 1$ when $H$ is a triangle with two disjoint pendant edges. Hor\v{n}\'ak, Jendrol', Schiermeyer, and Sot\'ak \cite{10.1002/jgt.21803} further proved that  
$\mathrm{rb}(W_{d}, C_4) = \left\lfloor \frac{4d}{3} \right\rfloor + 1$ and  
$\mathrm{rb}(W_{d}, C_5) = \left\lfloor \frac{3d}{2} \right\rfloor + 1$.  
Lan, Shi, and Song \cite{LAN20193216} subsequently determined 
$\mathrm{rb}(W_{d}, C_6) = \left\lfloor \frac{5d}{3} \right\rfloor + 1$ and 
established bounds for $\mathrm{rb}(W_{d}, C_k)$ as follows:
\begin{thm} \label{oldtheorem}
    For integers $d,k$ with $d\ge k-1$ and $k\ge 5$, $\left\lfloor \frac{2k-7}{k-3}d \right\rfloor+1 \le {\rm rb}(W_{d},C_k) \le \left\lfloor \frac{2k-5}{k-2}d \right\rfloor+1$.
\end{thm}
\noindent 
Recently, Jakhar, Budden, and Moun \cite{jakhar2025rainbow} proved that for all integers $d \geq 3$,  
$\mathrm{rb}(W_{d}, C^+_4) = \left\lfloor \frac{3d}{2} \right\rfloor + 1$, and for all $d \geq 5$,  
$\mathrm{rb}(W_{d}, C^+_5) = \left\lfloor \frac{3d}{2} \right\rfloor + 1$, where $C^+_t$ represents a $t$-cycle with an arbitrary chord.
Moreover, they established that $\mathrm{rb}(W^*_{d}, C_3) = \left\lfloor \frac{3d}{2} \right\rfloor + 1$ and  
$\left\lfloor \frac{4d}{3} \right\rfloor + 2 \leq \mathrm{rb}(W^*_{d}, C_4) \leq 2d + 1$, where $W^*_{d}$ is a double-hubbed wheels, i.e, the join of $2K_1$ and $C_d$. These findings elucidate how the local wheel structures embedded in $T_n$ directly constrain the rainbow number estimates for cycle-related subgraphs.

This paper deepens the research on rainbow numbers pertaining to cycle-related graphs within wheel graphs or wheel-related graphs by broadening the range of graph classes under investigation and enhancing the established extremal bounds.

The \textit{$t$-fan} graph $F_t$ is constructed from a $t$-cycle $C_t$ with vertices labeled $v_1,v_2,\dots,v_t$ in sequential order (forming the \emph{boundary cycle}) by designating $v_1$ as the \emph{central vertex} and adding all chords of the form $v_1v_i$ for $3 \leq i \leq t-1$; for brevity, we denote this construction as $F_t := [v_1; v_2\cdots v_t]$. Let $\Theta_{t,\ell}$ denote a class of graphs. For any graph $\theta_{t,\ell}$ within this class, it consists of a $t$-cycle along with $\ell$ chords, and all these $\ell$ chords originate from the same vertex. When $\ell = t - 3$, the graph $F_t$ and the graph $\theta_{t,\ell}$ represent the identical structure, specifically, a $t$-cycle with all feasible chords emanating from a single vertex. 
When $\ell=0$, the graph $\theta_{t,\ell}$ is exactly a $t$-cycle $C_t$.
The \emph{wheel graph} $W_d$ consists of a central \textit{hub} vertex connected to all vertices of a $d$-cycle $C_d$. The \emph{$s$-hubbed wheel} $W_d(s)$ is obtained by ``blowing up'' the hub of $W_d$ into $s$ independent vertices (i.e., replacing the single hub with $s$ non-adjacent vertices, each connected to all vertices of the original $d$-cycle); we call these $s$ vertices the \emph{hub} or \emph{hub vertices} and these $d$ vertices of the cycle \emph{boundary vertices} of $W_d(s)$. In $W_d(s)$, edges incident to one hub vertex are called \textit{spokes}, while all other edges are referred to as \textit{rim edges}. Figure~\ref{fig:A} provides concrete instances that exemplify the concepts introduced here.

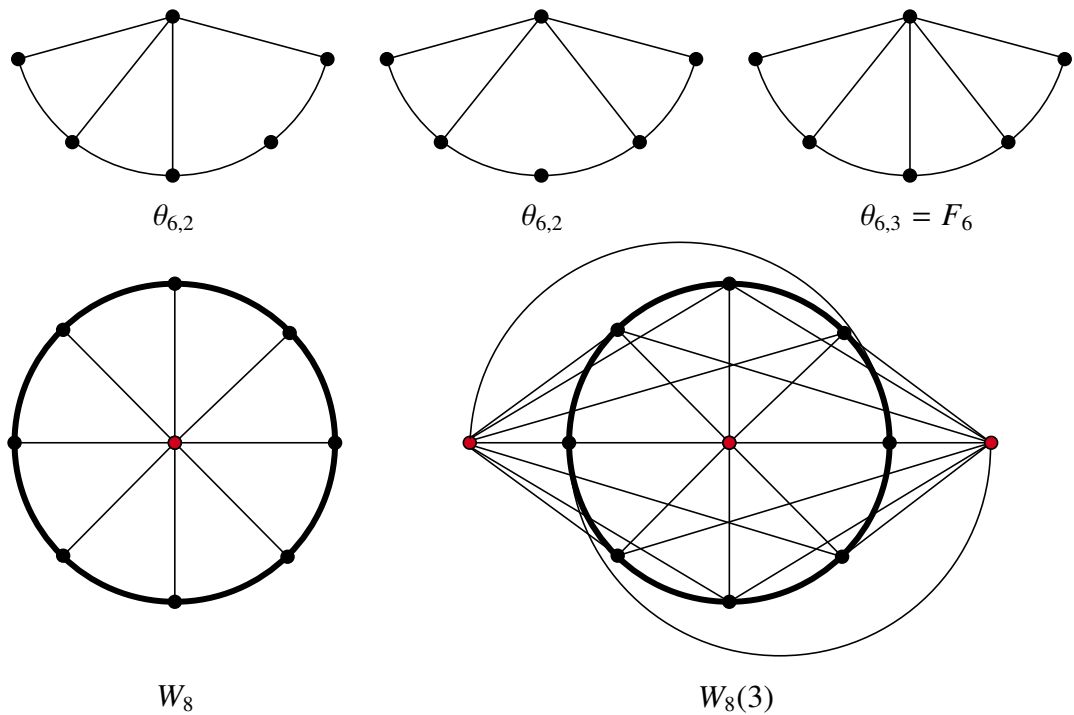
\begin{figure}[h]
    \centering

\tikzset{every picture/.style={line width=0.6pt}} %set default line width to 0.75pt        

\begin{tikzpicture}[x=0.6pt,y=0.6pt,yscale=-1,xscale=1]
%uncomment if require: \path (0,545); %set diagram left start at 0, and has height of 545

%Shape: Arc [id:dp05141195621185668] 
\draw  [draw opacity=0] (205.88,49.16) .. controls (194.38,91.7) and (155.51,123) .. (109.33,123) .. controls (63.38,123) and (24.67,92.01) .. (12.96,49.78) -- (109.33,23) -- cycle ; \draw   (205.88,49.16) .. controls (194.38,91.7) and (155.51,123) .. (109.33,123) .. controls (63.38,123) and (24.67,92.01) .. (12.96,49.78) ;  
%Flowchart: Connector [id:dp2313252402767969] 
\draw  [fill={rgb, 255:red, 0; green, 0; blue, 0 }  ,fill opacity=1 ] (105.33,23) .. controls (105.33,20.79) and (107.12,19) .. (109.33,19) .. controls (111.54,19) and (113.33,20.79) .. (113.33,23) .. controls (113.33,25.21) and (111.54,27) .. (109.33,27) .. controls (107.12,27) and (105.33,25.21) .. (105.33,23) -- cycle ;
%Flowchart: Connector [id:dp14579250841283442] 
\draw  [fill={rgb, 255:red, 0; green, 0; blue, 0 }  ,fill opacity=1 ] (105.33,123) .. controls (105.33,120.79) and (107.12,119) .. (109.33,119) .. controls (111.54,119) and (113.33,120.79) .. (113.33,123) .. controls (113.33,125.21) and (111.54,127) .. (109.33,127) .. controls (107.12,127) and (105.33,125.21) .. (105.33,123) -- cycle ;
%Flowchart: Connector [id:dp6199946437194452] 
\draw  [fill={rgb, 255:red, 0; green, 0; blue, 0 }  ,fill opacity=1 ] (201.88,49.8) .. controls (201.88,47.59) and (203.67,45.8) .. (205.88,45.8) .. controls (208.09,45.8) and (209.88,47.59) .. (209.88,49.8) .. controls (209.88,52.01) and (208.09,53.8) .. (205.88,53.8) .. controls (203.67,53.8) and (201.88,52.01) .. (201.88,49.8) -- cycle ;
%Flowchart: Connector [id:dp589714622766814] 
\draw  [fill={rgb, 255:red, 0; green, 0; blue, 0 }  ,fill opacity=1 ] (42.67,102) .. controls (42.67,99.79) and (44.46,98) .. (46.67,98) .. controls (48.88,98) and (50.67,99.79) .. (50.67,102) .. controls (50.67,104.21) and (48.88,106) .. (46.67,106) .. controls (44.46,106) and (42.67,104.21) .. (42.67,102) -- cycle ;
%Flowchart: Connector [id:dp7302429765377079] 
\draw  [fill={rgb, 255:red, 0; green, 0; blue, 0 }  ,fill opacity=1 ] (8.96,49.78) .. controls (8.96,47.57) and (10.75,45.78) .. (12.96,45.78) .. controls (15.17,45.78) and (16.96,47.57) .. (16.96,49.78) .. controls (16.96,51.99) and (15.17,53.78) .. (12.96,53.78) .. controls (10.75,53.78) and (8.96,51.99) .. (8.96,49.78) -- cycle ;
%Straight Lines [id:da2770316700101616] 
\draw    (109.33,23) -- (109.33,127) ;
%Straight Lines [id:da2979053531167679] 
\draw    (12.96,49.78) -- (109.33,23) ;
%Straight Lines [id:da14952451450154647] 
\draw    (109.33,23) -- (205.88,49.16) ;
%Flowchart: Connector [id:dp6119309087370055] 
\draw  [fill={rgb, 255:red, 0; green, 0; blue, 0 }  ,fill opacity=1 ] (166.7,102) .. controls (166.7,99.79) and (168.49,98) .. (170.7,98) .. controls (172.91,98) and (174.7,99.79) .. (174.7,102) .. controls (174.7,104.21) and (172.91,106) .. (170.7,106) .. controls (168.49,106) and (166.7,104.21) .. (166.7,102) -- cycle ;
%Straight Lines [id:da9144846461148459] 
\draw    (109.33,23) -- (46.67,102) ;
%Shape: Arc [id:dp7011033254719767] 
\draw  [draw opacity=0] (435.92,49.16) .. controls (424.42,91.7) and (385.55,123) .. (339.37,123) .. controls (293.42,123) and (254.71,92.01) .. (243,49.78) -- (339.37,23) -- cycle ; \draw   (435.92,49.16) .. controls (424.42,91.7) and (385.55,123) .. (339.37,123) .. controls (293.42,123) and (254.71,92.01) .. (243,49.78) ;  
%Flowchart: Connector [id:dp3364127248522335] 
\draw  [fill={rgb, 255:red, 0; green, 0; blue, 0 }  ,fill opacity=1 ] (335.37,23) .. controls (335.37,20.79) and (337.16,19) .. (339.37,19) .. controls (341.58,19) and (343.37,20.79) .. (343.37,23) .. controls (343.37,25.21) and (341.58,27) .. (339.37,27) .. controls (337.16,27) and (335.37,25.21) .. (335.37,23) -- cycle ;
%Flowchart: Connector [id:dp5928931770087] 
\draw  [fill={rgb, 255:red, 0; green, 0; blue, 0 }  ,fill opacity=1 ] (335.37,123) .. controls (335.37,120.79) and (337.16,119) .. (339.37,119) .. controls (341.58,119) and (343.37,120.79) .. (343.37,123) .. controls (343.37,125.21) and (341.58,127) .. (339.37,127) .. controls (337.16,127) and (335.37,125.21) .. (335.37,123) -- cycle ;
%Flowchart: Connector [id:dp7748177086521708] 
\draw  [fill={rgb, 255:red, 0; green, 0; blue, 0 }  ,fill opacity=1 ] (431.92,49.8) .. controls (431.92,47.59) and (433.71,45.8) .. (435.92,45.8) .. controls (438.13,45.8) and (439.92,47.59) .. (439.92,49.8) .. controls (439.92,52.01) and (438.13,53.8) .. (435.92,53.8) .. controls (433.71,53.8) and (431.92,52.01) .. (431.92,49.8) -- cycle ;
%Flowchart: Connector [id:dp49469234758302805] 
\draw  [fill={rgb, 255:red, 0; green, 0; blue, 0 }  ,fill opacity=1 ] (272.71,102) .. controls (272.71,99.79) and (274.5,98) .. (276.71,98) .. controls (278.92,98) and (280.71,99.79) .. (280.71,102) .. controls (280.71,104.21) and (278.92,106) .. (276.71,106) .. controls (274.5,106) and (272.71,104.21) .. (272.71,102) -- cycle ;
%Flowchart: Connector [id:dp32740484698807304] 
\draw  [fill={rgb, 255:red, 0; green, 0; blue, 0 }  ,fill opacity=1 ] (239,49.78) .. controls (239,47.57) and (240.79,45.78) .. (243,45.78) .. controls (245.21,45.78) and (247,47.57) .. (247,49.78) .. controls (247,51.99) and (245.21,53.78) .. (243,53.78) .. controls (240.79,53.78) and (239,51.99) .. (239,49.78) -- cycle ;
%Straight Lines [id:da022374536473311002] 
\draw    (243,49.78) -- (339.37,23) ;
%Straight Lines [id:da06349926658159011] 
\draw    (339.37,23) -- (435.92,49.16) ;
%Flowchart: Connector [id:dp2262884583709175] 
\draw  [fill={rgb, 255:red, 0; green, 0; blue, 0 }  ,fill opacity=1 ] (396.74,102) .. controls (396.74,99.79) and (398.53,98) .. (400.74,98) .. controls (402.95,98) and (404.74,99.79) .. (404.74,102) .. controls (404.74,104.21) and (402.95,106) .. (400.74,106) .. controls (398.53,106) and (396.74,104.21) .. (396.74,102) -- cycle ;
%Straight Lines [id:da5330091276485629] 
\draw    (339.37,23) -- (276.71,102) ;
%Shape: Arc [id:dp23455893778038295] 
\draw  [draw opacity=0] (665.92,49.16) .. controls (654.42,91.7) and (615.55,123) .. (569.37,123) .. controls (523.42,123) and (484.71,92.01) .. (473,49.78) -- (569.37,23) -- cycle ; \draw   (665.92,49.16) .. controls (654.42,91.7) and (615.55,123) .. (569.37,123) .. controls (523.42,123) and (484.71,92.01) .. (473,49.78) ;  
%Flowchart: Connector [id:dp3790442848072222] 
\draw  [fill={rgb, 255:red, 0; green, 0; blue, 0 }  ,fill opacity=1 ] (565.37,23) .. controls (565.37,20.79) and (567.16,19) .. (569.37,19) .. controls (571.58,19) and (573.37,20.79) .. (573.37,23) .. controls (573.37,25.21) and (571.58,27) .. (569.37,27) .. controls (567.16,27) and (565.37,25.21) .. (565.37,23) -- cycle ;
%Flowchart: Connector [id:dp5410235931734424] 
\draw  [fill={rgb, 255:red, 0; green, 0; blue, 0 }  ,fill opacity=1 ] (565.37,123) .. controls (565.37,120.79) and (567.16,119) .. (569.37,119) .. controls (571.58,119) and (573.37,120.79) .. (573.37,123) .. controls (573.37,125.21) and (571.58,127) .. (569.37,127) .. controls (567.16,127) and (565.37,125.21) .. (565.37,123) -- cycle ;
%Flowchart: Connector [id:dp043802962055113026] 
\draw  [fill={rgb, 255:red, 0; green, 0; blue, 0 }  ,fill opacity=1 ] (661.92,49.8) .. controls (661.92,47.59) and (663.71,45.8) .. (665.92,45.8) .. controls (668.13,45.8) and (669.92,47.59) .. (669.92,49.8) .. controls (669.92,52.01) and (668.13,53.8) .. (665.92,53.8) .. controls (663.71,53.8) and (661.92,52.01) .. (661.92,49.8) -- cycle ;
%Flowchart: Connector [id:dp19576903603505835] 
\draw  [fill={rgb, 255:red, 0; green, 0; blue, 0 }  ,fill opacity=1 ] (502.71,102) .. controls (502.71,99.79) and (504.5,98) .. (506.71,98) .. controls (508.92,98) and (510.71,99.79) .. (510.71,102) .. controls (510.71,104.21) and (508.92,106) .. (506.71,106) .. controls (504.5,106) and (502.71,104.21) .. (502.71,102) -- cycle ;
%Flowchart: Connector [id:dp7178112933927929] 
\draw  [fill={rgb, 255:red, 0; green, 0; blue, 0 }  ,fill opacity=1 ] (469,49.78) .. controls (469,47.57) and (470.79,45.78) .. (473,45.78) .. controls (475.21,45.78) and (477,47.57) .. (477,49.78) .. controls (477,51.99) and (475.21,53.78) .. (473,53.78) .. controls (470.79,53.78) and (469,51.99) .. (469,49.78) -- cycle ;
%Straight Lines [id:da011124673777457073] 
\draw    (569.37,23) -- (569.37,127) ;
%Straight Lines [id:da4945460313495078] 
\draw    (473,49.78) -- (569.37,23) ;
%Straight Lines [id:da5307564714968955] 
\draw    (569.37,23) -- (665.92,49.16) ;
%Flowchart: Connector [id:dp7124327789728431] 
\draw  [fill={rgb, 255:red, 0; green, 0; blue, 0 }  ,fill opacity=1 ] (626.74,102) .. controls (626.74,99.79) and (628.53,98) .. (630.74,98) .. controls (632.95,98) and (634.74,99.79) .. (634.74,102) .. controls (634.74,104.21) and (632.95,106) .. (630.74,106) .. controls (628.53,106) and (626.74,104.21) .. (626.74,102) -- cycle ;
%Straight Lines [id:da9773705953461163] 
\draw    (569.37,23) -- (506.71,102) ;
%Straight Lines [id:da4435836844377212] 
\draw    (339.37,23) -- (400.74,102) ;
%Straight Lines [id:da39693593863896903] 
\draw    (569.37,23) -- (630.74,102) ;
%Shape: Circle [id:dp6175478597528898] 
\draw  [color={rgb, 255:red, 0; green, 0; blue, 0 }  ,draw opacity=1 ][line width=2.25]  (10.67,291) .. controls (10.67,235.77) and (55.44,191) .. (110.67,191) .. controls (165.9,191) and (210.67,235.77) .. (210.67,291) .. controls (210.67,346.23) and (165.9,391) .. (110.67,391) .. controls (55.44,391) and (10.67,346.23) .. (10.67,291) -- cycle ;
%Straight Lines [id:da05291599450902473] 
\draw    (110.67,291) -- (210.67,291) ;
%Flowchart: Connector [id:dp6841359476922972] 
\draw  [fill={rgb, 255:red, 0; green, 0; blue, 0 }  ,fill opacity=1 ] (106.67,291) .. controls (106.67,288.79) and (108.46,287) .. (110.67,287) .. controls (112.88,287) and (114.67,288.79) .. (114.67,291) .. controls (114.67,293.21) and (112.88,295) .. (110.67,295) .. controls (108.46,295) and (106.67,293.21) .. (106.67,291) -- cycle ;
%Straight Lines [id:da7162647504788435] 
\draw    (110.67,291) -- (110.67,391) ;
%Straight Lines [id:da6008591083943959] 
\draw    (110.67,191) -- (110.67,291) ;
%Straight Lines [id:da38728965297528073] 
\draw    (10.67,291) -- (110.67,291) ;
%Straight Lines [id:da061308746616873355] 
\draw    (41,220) -- (110.67,291) ;
%Straight Lines [id:da26994828488865297] 
\draw    (110.67,291) -- (181,362.67) ;
%Straight Lines [id:da8812830656299073] 
\draw [fill={rgb, 255:red, 208; green, 2; blue, 27 }  ,fill opacity=1 ]   (110.67,291) -- (113.03,288.72) -- (182.33,222) ;
%Straight Lines [id:da9464841406396582] 
\draw    (110.67,291) -- (41,362) ;
%Flowchart: Connector [id:dp7528931825883303] 
\draw  [fill={rgb, 255:red, 0; green, 0; blue, 0 }  ,fill opacity=1 ] (106.67,191) .. controls (106.67,188.79) and (108.46,187) .. (110.67,187) .. controls (112.88,187) and (114.67,188.79) .. (114.67,191) .. controls (114.67,193.21) and (112.88,195) .. (110.67,195) .. controls (108.46,195) and (106.67,193.21) .. (106.67,191) -- cycle ;
%Flowchart: Connector [id:dp806613902650847] 
\draw  [fill={rgb, 255:red, 0; green, 0; blue, 0 }  ,fill opacity=1 ] (37,220) .. controls (37,217.79) and (38.79,216) .. (41,216) .. controls (43.21,216) and (45,217.79) .. (45,220) .. controls (45,222.21) and (43.21,224) .. (41,224) .. controls (38.79,224) and (37,222.21) .. (37,220) -- cycle ;
%Flowchart: Connector [id:dp4495793217406616] 
\draw  [fill={rgb, 255:red, 0; green, 0; blue, 0 }  ,fill opacity=1 ] (6.67,291) .. controls (6.67,288.79) and (8.46,287) .. (10.67,287) .. controls (12.88,287) and (14.67,288.79) .. (14.67,291) .. controls (14.67,293.21) and (12.88,295) .. (10.67,295) .. controls (8.46,295) and (6.67,293.21) .. (6.67,291) -- cycle ;
%Flowchart: Connector [id:dp059911925166102] 
\draw  [fill={rgb, 255:red, 0; green, 0; blue, 0 }  ,fill opacity=1 ] (37,362) .. controls (37,359.79) and (38.79,358) .. (41,358) .. controls (43.21,358) and (45,359.79) .. (45,362) .. controls (45,364.21) and (43.21,366) .. (41,366) .. controls (38.79,366) and (37,364.21) .. (37,362) -- cycle ;
%Flowchart: Connector [id:dp3402738821012563] 
\draw  [fill={rgb, 255:red, 0; green, 0; blue, 0 }  ,fill opacity=1 ] (106.67,391) .. controls (106.67,388.79) and (108.46,387) .. (110.67,387) .. controls (112.88,387) and (114.67,388.79) .. (114.67,391) .. controls (114.67,393.21) and (112.88,395) .. (110.67,395) .. controls (108.46,395) and (106.67,393.21) .. (106.67,391) -- cycle ;
%Flowchart: Connector [id:dp7668012598084153] 
\draw  [fill={rgb, 255:red, 0; green, 0; blue, 0 }  ,fill opacity=1 ] (177,362.67) .. controls (177,360.46) and (178.79,358.67) .. (181,358.67) .. controls (183.21,358.67) and (185,360.46) .. (185,362.67) .. controls (185,364.88) and (183.21,366.67) .. (181,366.67) .. controls (178.79,366.67) and (177,364.88) .. (177,362.67) -- cycle ;
%Flowchart: Connector [id:dp6249603050815784] 
\draw  [fill={rgb, 255:red, 0; green, 0; blue, 0 }  ,fill opacity=1 ] (206.67,291) .. controls (206.67,288.79) and (208.46,287) .. (210.67,287) .. controls (212.88,287) and (214.67,288.79) .. (214.67,291) .. controls (214.67,293.21) and (212.88,295) .. (210.67,295) .. controls (208.46,295) and (206.67,293.21) .. (206.67,291) -- cycle ;
%Flowchart: Connector [id:dp10080347030181369] 
\draw  [fill={rgb, 255:red, 0; green, 0; blue, 0 }  ,fill opacity=1 ] (178.33,222) .. controls (178.33,219.79) and (180.12,218) .. (182.33,218) .. controls (184.54,218) and (186.33,219.79) .. (186.33,222) .. controls (186.33,224.21) and (184.54,226) .. (182.33,226) .. controls (180.12,226) and (178.33,224.21) .. (178.33,222) -- cycle ;
%Shape: Circle [id:dp33092601163471436] 
\draw  [line width=2.25]  (356.67,291) .. controls (356.67,235.77) and (401.44,191) .. (456.67,191) .. controls (511.9,191) and (556.67,235.77) .. (556.67,291) .. controls (556.67,346.23) and (511.9,391) .. (456.67,391) .. controls (401.44,391) and (356.67,346.23) .. (356.67,291) -- cycle ;
%Straight Lines [id:da5427736905973484] 
\draw    (456.67,291) -- (556.67,291) ;
%Flowchart: Connector [id:dp3073137481275443] 
\draw  [fill={rgb, 255:red, 0; green, 0; blue, 0 }  ,fill opacity=1 ] (452.67,291) .. controls (452.67,288.79) and (454.46,287) .. (456.67,287) .. controls (458.88,287) and (460.67,288.79) .. (460.67,291) .. controls (460.67,293.21) and (458.88,295) .. (456.67,295) .. controls (454.46,295) and (452.67,293.21) .. (452.67,291) -- cycle ;
%Straight Lines [id:da7395068970783611] 
\draw    (456.67,291) -- (456.67,391) ;
%Straight Lines [id:da9854575468830553] 
\draw    (456.67,191) -- (456.67,291) ;
%Straight Lines [id:da16889708217484078] 
\draw    (356.67,291) -- (456.67,291) ;
%Straight Lines [id:da8325781833147479] 
\draw    (387,220) -- (456.67,291) ;
%Straight Lines [id:da1997939111802256] 
\draw    (456.67,291) -- (527,362.67) ;
%Straight Lines [id:da0761280010829457] 
\draw    (456.67,291) -- (528.33,222) ;
%Straight Lines [id:da4658402610110737] 
\draw    (456.67,291) -- (387,362) ;
%Flowchart: Connector [id:dp05927174308463323] 
\draw  [fill={rgb, 255:red, 0; green, 0; blue, 0 }  ,fill opacity=1 ] (452.67,191) .. controls (452.67,188.79) and (454.46,187) .. (456.67,187) .. controls (458.88,187) and (460.67,188.79) .. (460.67,191) .. controls (460.67,193.21) and (458.88,195) .. (456.67,195) .. controls (454.46,195) and (452.67,193.21) .. (452.67,191) -- cycle ;
%Flowchart: Connector [id:dp6675834444097568] 
\draw  [fill={rgb, 255:red, 0; green, 0; blue, 0 }  ,fill opacity=1 ] (383,220) .. controls (383,217.79) and (384.79,216) .. (387,216) .. controls (389.21,216) and (391,217.79) .. (391,220) .. controls (391,222.21) and (389.21,224) .. (387,224) .. controls (384.79,224) and (383,222.21) .. (383,220) -- cycle ;
%Flowchart: Connector [id:dp1434796459245058] 
\draw  [fill={rgb, 255:red, 0; green, 0; blue, 0 }  ,fill opacity=1 ] (352.67,291) .. controls (352.67,288.79) and (354.46,287) .. (356.67,287) .. controls (358.88,287) and (360.67,288.79) .. (360.67,291) .. controls (360.67,293.21) and (358.88,295) .. (356.67,295) .. controls (354.46,295) and (352.67,293.21) .. (352.67,291) -- cycle ;
%Flowchart: Connector [id:dp6909004740941211] 
\draw  [fill={rgb, 255:red, 0; green, 0; blue, 0 }  ,fill opacity=1 ] (383,362) .. controls (383,359.79) and (384.79,358) .. (387,358) .. controls (389.21,358) and (391,359.79) .. (391,362) .. controls (391,364.21) and (389.21,366) .. (387,366) .. controls (384.79,366) and (383,364.21) .. (383,362) -- cycle ;
%Flowchart: Connector [id:dp0027832485864272005] 
\draw  [fill={rgb, 255:red, 0; green, 0; blue, 0 }  ,fill opacity=1 ] (452.67,391) .. controls (452.67,388.79) and (454.46,387) .. (456.67,387) .. controls (458.88,387) and (460.67,388.79) .. (460.67,391) .. controls (460.67,393.21) and (458.88,395) .. (456.67,395) .. controls (454.46,395) and (452.67,393.21) .. (452.67,391) -- cycle ;
%Flowchart: Connector [id:dp7170927561737799] 
\draw  [fill={rgb, 255:red, 0; green, 0; blue, 0 }  ,fill opacity=1 ] (523,362.67) .. controls (523,360.46) and (524.79,358.67) .. (527,358.67) .. controls (529.21,358.67) and (531,360.46) .. (531,362.67) .. controls (531,364.88) and (529.21,366.67) .. (527,366.67) .. controls (524.79,366.67) and (523,364.88) .. (523,362.67) -- cycle ;
%Flowchart: Connector [id:dp04497441493397414] 
\draw  [fill={rgb, 255:red, 0; green, 0; blue, 0 }  ,fill opacity=1 ] (552.67,291) .. controls (552.67,288.79) and (554.46,287) .. (556.67,287) .. controls (558.88,287) and (560.67,288.79) .. (560.67,291) .. controls (560.67,293.21) and (558.88,295) .. (556.67,295) .. controls (554.46,295) and (552.67,293.21) .. (552.67,291) -- cycle ;
%Flowchart: Connector [id:dp8665019729002787] 
\draw  [fill={rgb, 255:red, 0; green, 0; blue, 0 }  ,fill opacity=1 ] (524.33,222) .. controls (524.33,219.79) and (526.12,218) .. (528.33,218) .. controls (530.54,218) and (532.33,219.79) .. (532.33,222) .. controls (532.33,224.21) and (530.54,226) .. (528.33,226) .. controls (526.12,226) and (524.33,224.21) .. (524.33,222) -- cycle ;
%Straight Lines [id:da18736493229152895] 
\draw    (290.71,291) -- (356.67,291) ;
%Straight Lines [id:da44958978018452966] 
\draw    (387,220) -- (290.71,291) ;
%Straight Lines [id:da5284390025803614] 
\draw    (290.71,291) -- (387,362) ;
%Straight Lines [id:da777938219830655] 
\draw    (290.71,291) -- (456.67,391) ;
%Straight Lines [id:da49038959897691936] 
\draw    (290.71,291) -- (456.67,191) ;
%Straight Lines [id:da9842922310797302] 
\draw    (290.71,291) -- (528.33,222) ;
%Straight Lines [id:da8495217577887713] 
\draw    (290.71,291) -- (527,362.67) ;
%Straight Lines [id:da09032476759687658] 
\draw    (556.67,291) -- (620.04,291) ;
%Straight Lines [id:da8719640052540123] 
\draw    (528.33,222) -- (620.04,291) ;
%Straight Lines [id:da9450434256803024] 
\draw    (620.04,291) -- (527,362.67) ;
%Straight Lines [id:da8718734986229] 
\draw    (620.04,291) -- (456.67,391) ;
%Straight Lines [id:da6059374397829118] 
\draw    (456.67,191) -- (620.04,291) ;
%Flowchart: Connector [id:dp13257554116708437] 
\draw  [fill={rgb, 255:red, 208; green, 2; blue, 27 }  ,fill opacity=1 ] (106.67,291) .. controls (106.67,288.79) and (108.46,287) .. (110.67,287) .. controls (112.88,287) and (114.67,288.79) .. (114.67,291) .. controls (114.67,293.21) and (112.88,295) .. (110.67,295) .. controls (108.46,295) and (106.67,293.21) .. (106.67,291) -- cycle ;
%Flowchart: Connector [id:dp9332119324274559] 
\draw  [fill={rgb, 255:red, 208; green, 2; blue, 27 }  ,fill opacity=1 ] (290.71,291) .. controls (290.71,288.79) and (292.5,287) .. (294.71,287) .. controls (296.92,287) and (298.71,288.79) .. (298.71,291) .. controls (298.71,293.21) and (296.92,295) .. (294.71,295) .. controls (292.5,295) and (290.71,293.21) .. (290.71,291) -- cycle ;
%Flowchart: Connector [id:dp045026002264085374] 
\draw  [fill={rgb, 255:red, 208; green, 2; blue, 27 }  ,fill opacity=1 ] (452.67,291) .. controls (452.67,288.79) and (454.46,287) .. (456.67,287) .. controls (458.88,287) and (460.67,288.79) .. (460.67,291) .. controls (460.67,293.21) and (458.88,295) .. (456.67,295) .. controls (454.46,295) and (452.67,293.21) .. (452.67,291) -- cycle ;
%Shape: Arc [id:dp6083254742710491] 
\draw  [draw opacity=0] (294.98,288.6) .. controls (298.27,221.23) and (353.46,166.78) .. (422.47,165.02) .. controls (494.76,163.17) and (554.85,219.86) .. (556.68,291.64) .. controls (556.72,293.17) and (556.73,294.71) .. (556.72,296.23) -- (425.78,294.97) -- cycle ; \draw   (294.98,288.6) .. controls (298.27,221.23) and (353.46,166.78) .. (422.47,165.02) .. controls (494.76,163.17) and (554.85,219.86) .. (556.68,291.64) .. controls (556.72,293.17) and (556.73,294.71) .. (556.72,296.23) ;  
%Shape: Arc [id:dp2122802725661923] 
\draw  [draw opacity=0] (619.65,293.83) .. controls (619.44,366.28) and (560.65,424.94) .. (488.15,424.94) .. controls (415.53,424.94) and (356.65,366.07) .. (356.65,293.44) .. controls (356.65,292.79) and (356.66,292.14) .. (356.67,291.49) -- (488.15,293.44) -- cycle ; \draw   (619.65,293.83) .. controls (619.44,366.28) and (560.65,424.94) .. (488.15,424.94) .. controls (415.53,424.94) and (356.65,366.07) .. (356.65,293.44) .. controls (356.65,292.79) and (356.66,292.14) .. (356.67,291.49) ;  
%Straight Lines [id:da6964338487450995] 
\draw    (387,220) -- (620.04,291) ;
%Straight Lines [id:da17404614445036815] 
\draw    (387,362) -- (620.04,291) ;
%Flowchart: Connector [id:dp5895386667767795] 
\draw  [fill={rgb, 255:red, 208; green, 2; blue, 27 }  ,fill opacity=1 ] (616.04,291) .. controls (616.04,288.79) and (617.83,287) .. (620.04,287) .. controls (622.25,287) and (624.04,288.79) .. (624.04,291) .. controls (624.04,293.21) and (622.25,295) .. (620.04,295) .. controls (617.83,295) and (616.04,293.21) .. (616.04,291) -- cycle ;

% Text Node
\draw (95.33,139.4) node [anchor=north west][inner sep=0.75pt]    {$\theta _{6,2}$};
% Text Node
\draw (325.33,139.4) node [anchor=north west][inner sep=0.75pt]    {$\theta _{6,2}$};
% Text Node
\draw (536.67,139.4) node [anchor=north west][inner sep=0.75pt]    {$\theta _{6,3} =F_{6}$};
% Text Node
\draw (97.33,441.7) node [anchor=north west][inner sep=0.75pt]    {$W_{8}$};
% Text Node
\draw (435.33,441.73) node [anchor=north west][inner sep=0.75pt]    {$W_{8}( 3)$};

\end{tikzpicture}

    \caption{Examples of $\theta_{6,2},F_6,W_8$ and $W_8(3)$.}
    \label{fig:A}
\end{figure}

The following are two main theorems of this paper.

%\begin{restatable*}{thm}{mainone}\label{thm:main-1}
\begin{thm}\label{thm:main-1}
    Let \( d, s, t \) be integers such that \( d \geq 3t - 5 \). The  rainbow number of \( F_t \) in $W_{d}(s)$ is given by
\[
{\rm rb}(W_{d}(s), F_t) = 
\begin{cases}
\left\lfloor \dfrac{2t - 5}{t - 2} d\right\rfloor + 1, & \text{if } s = 1 \text{ and } t\ge 4, \\[10pt]
\left\lfloor \dfrac{3t-10}{t - 3}d \right\rfloor + 1, & \text{if } s = 2\text{ and } t\ge 6.\\[10pt]
\left\lfloor \dfrac{(s + 1)t - (3s + 4)}{t - 3}d \right\rfloor + 1, & \text{if } s \geq 3\text{ and } t\ge 7.
\end{cases}
\]
\end{thm}
%\end{restatable*}

%\begin{restatable*}{thm}{maintwo}\label{thm:main-2}
\begin{thm}\label{thm:main-2}
 For integers $d,\ell$ and $t$ such that $d\geq 3t-5 $ and $t\ge \max\{5, \ell+4\}$, we have
    $${\rm rb}(W_d ,\theta_{t,\ell}) = \bigg\lfloor\frac{2t-7}{t-3}d\bigg\rfloor +1$$
    for every $\theta_{t,\ell}\in \Theta_{t,\ell}$.
\end{thm}
%\end{restatable*}

The condition \(t \geq 4\) stated in the first assertion of Theorem~\ref{thm:main-1} is indeed necessary. This necessity stems from the fact that \(F_3 = C_3\). As demonstrated by Qin, Lei, and Li in their work~\cite{QIN2020124888}, it has been proved that \(\mathrm{rb}(W_{d}, C_3) = d + 2\). 
The conditions \(t \geq 6\) and \(t \geq 7\) stated in the second and third assertions of Theorem~\ref{thm:main-1} are also essential, as we will show
\begin{align*}
    {\rm rb}(W_{d}(2), F_5)&\leq \left\lfloor \frac{33}{14}d \right\rfloor+1< \left\lfloor \frac{5}{2}d \right\rfloor+1 \\
    {\rm rb}(W_{d}(3), F_6)&\leq \left\lfloor \frac{55}{16}d \right\rfloor+1< \left\lfloor \frac{11}{3}d \right\rfloor+1 
\end{align*}
by Theorems \ref{thm5} in Section \ref{sec:final}.

Combining Theorem \ref{thm:main-1} (with parameters $s=1$ and $t=4$) and Theorem \ref{thm:main-2} (with $t=5$ and $\ell=1$), we obtain the exact rainbow numbers:
$$
\mathrm{rb}(W_{d}, C^+_4) = \mathrm{rb}(W_{d}, C^+_5) = \left\lfloor \frac{3d}{2} \right\rfloor + 1,
$$
which coincides with the recent result by Jakhar, Budden, and Moun \cite{jakhar2025rainbow}, who further proposed extending the study to  large cycles and large chorded cycles in wheel graphs. Interestingly, by specializing Theorem \ref{thm:main-2} to the cases $\ell=0$ and $\ell=1$, we derive the following.

\begin{cor}\label{thm:main-3}
    For integers $d$ and $t$ such that $d\geq 3t-5 $ and $t\ge 5$, we have
    $${\rm rb}(W_d ,C_t) = {\rm rb}(W_d ,C^+_t)= \bigg\lfloor\frac{2t-7}{t-3}d\bigg\rfloor +1.$$
\end{cor}
\noindent 
This result establishes that the lower bound stated in Theorem \ref{oldtheorem} of Lan, Shi, and Song \cite{LAN20193216} for the rainbow number of $C_t$ ($t \geq 5$) in $W_d$ is in fact exact, coinciding precisely with the value derived in Corollary \ref{thm:main-3}. The conditions $t \geq 5$ in Theorem \ref{thm:main-2} and Corollary \ref{thm:main-3} are necessary. Note that $\theta_{4,1} = F_4$ and $\theta_{4,0} = C_4$. Theorem \ref{thm:main-1} implies  
${\rm rb}(W_d, F_4) = \lfloor \frac{3d}{2} \rfloor + 1$, and Hor\v{n}\'{a}k, Jendrol', Schiermeyer, and Sot\'{a}k \cite{10.1002/jgt.21803} proved that ${\rm rb}(W_d, C_4) = \lfloor \frac{4d}{3} \rfloor + 1$.

This paper is structured as follows. Section \ref{sec:countinglemma} establishes a counting argument for the number of $\theta_{t,\ell}$-subgraphs that cover specified edges in the $s$-hubbed wheel graph $W_d(s)$. This enumeration serves as a foundational tool for subsequent proofs. In Section \ref{sec:upperbound}, we derive upper bounds for both $\mathrm{rb}(W_d(s), F_t)$ and $\mathrm{rb}(W_d, \theta_{t,\ell})$, while corresponding lower bounds are presented in Section \ref{sec:lowerbound}. Finally, Section \ref{sec:final} summarizes our key findings and discusses potential directions for future research in this domain.

\section{Counting lemma} \label{sec:countinglemma}

% For a positive integer $t$ and an integer $x>-t$, set
% \[
% x \bmod^*t=
% \begin{cases}
% x \bmod t, &\text{if } x\geq 0\\
% x + t, &\text{if } x < 0
% \end{cases}
% \]

The structure of a graph $\theta_{t,\ell}$ is characterized as follows. There exists a $t$-cycle $C_t$, with its vertices labeled $v_1,v_2,\cdots,v_t$ in order. Additionally, there are $\ell$ chords selected from the set $\{v_1v_i\mid 3\leq i\leq t - 1\}$. We call this $C_t$ be the \textit{boundary cycle} of $\theta_{t,\ell}$.  For convenience, we represent $\theta_{t,\ell}$ by a vector $\mathbf{X}(\theta_{t,\ell}):=(2,i_1,i_2,\ldots,i_\ell,t)$ with 
$2<i_1< i_2< \cdots <i_\ell<t$, where $v_1v_{i_j}$ is an edge for each $j\in [\ell]$. If $\ell=0$ or $i_{j}+i_{\ell+1-j}=t +2$ for every $j\leq (\ell+1)/2$, then we call that $\theta_{t,\ell}$ or $\mathbf{X}(\theta_{t,\ell})$ is \textit{symmetric}. For example, the first subfigure in Figure~\ref{fig:A} depicts an asymmetric $\theta_{6,2}$, while the second subfigure shows an symmetric variant.

\begin{lemma} \label{lem:p2}
Let $d,s,t,\ell$ be nonnegative integers such that $d\geq 2t-3$ and $t\geq \max\{4,\ell+3\}$.
For any pair of edges in the graph $W_d (s)$, there exist at most $N_{\ref{lem:p2}}$ copies of the graph $\theta_{t,\ell}$ that contain both of these edges and exactly one hub vertex of $W_d(s)$, where
\[
N_{\ref{lem:p2}}=
\begin{cases}
\max\{s(t-3),t-2\}, &\text{if } \theta_{t,\ell} = F_t,\\
s(t - 3), &\text{if } \theta_{t,\ell} \not= F_t \text{ and } \mathbf{X}(\theta_{t,\ell}) \text{ is symmetric},\\
2s(t - 3), &\text{if } \mathbf{X}(\theta_{t,\ell}) \text{ is asymmetric.}
\end{cases}
\]
\end{lemma}

\begin{proof}
Let $e_1$ and $e_2$ be two fixed edges in the graph $W_d(s)$. We define $q$ as the quantity representing the number of copies of the graph $\theta_{t,\ell}$ that simultaneously contain both $e_1$ and $e_2$, and exactly one hub vertex of $W_d(s)$. 
The distribution of $e_1$ and $e_2$ must fall into exactly one of the following mutually exclusive cases
\begin{enumerate}[label={Case A\arabic*}]
\item \label{caseA1} $e_1\in E_{u_a}$ and $e_2\in E_{u_b}$ for some $a,b\in [s]$.\\
If $a\neq b$, then it is clear $q=0$. Thus, we assume $a= b$.
Let $e_1=u_av_i$ and $e_2=u_av_j$ ($i<j$). 
% Let $d(i,j)$ denote the minimal cycle distance between vertices $v_i$ and $v_j$ on the cycle $C_d$. Clearly,
% $d(i,j) \le \left\lfloor d/2 \right\rfloor$. 
If \( j - i >t-2 \) and \( (i - j) \bmod d > t-2 \), then $q=0$. 
Since $d\geq 2t-3$, it is impossible two inequalities \( j - i \leq t-2 \) and \( (i - j) \bmod d \leq t-2 \) to hold simultaneously; thus, without loss of generality, we assume $j-i\le t-2$.
Now, there are exactly $t-(j-i)-1$ copies of $F_t$ containing both $e_1$ and $e_2$; they are
$F_t^1=[u_a;v_i\cdots v_{i+t-2}]$,
$F_t^2=[u_a;v_{i-1}\cdots v_{i+t-3}]$,
$\ldots$, and $F_t^{t-(j-i)-1}=[u_a;v_{j+2-t} \ldots v_j]$ (note that $t-(j-i)-1\geq t-(t-2)-1=1$). 
If a negative number occurs as a subscript, replace it with the sum of this negative number and $t$. The same applies to subsequent similar cases.

If $\mathbf{X}(\theta_{t,\ell})$ is symmetric or asymmetric, then each $F_t^h$ contains one or two copies of $\theta_{t,\ell}$, respectively. 
If $F_t^h$ contains a copy of $\theta_{t,\ell}$ that covers $e_1$ and $e_2$, then 
\begin{enumerate}[label={(\roman*)}]
    \item \label{case1.1} $h+1$ and $\big(h+1+j-i\big) \bmod t$ are the components of vector $\mathbf{X}(\theta_{t,\ell})$, or
    \item \label{case1.2} $t-h+1-(j-i)$ and $(t-h+1) \bmod t$ are the components of vector $\mathbf{X}(\theta_{t,\ell})$.
\end{enumerate}
Therefore, if \ref{case1.1} holds for every $F_t^h$, then $2,3,\ldots,t$ are 
the components of vector $\mathbf{X}(\theta_{t,\ell})$, and thus 
$\theta_{t,\ell}=F_t$, which is symmetric. Similarly, if \ref{case1.2} holds for every $F_t^h$, then we come to the same conclusion. 
Thus, if $\theta_{t,\ell}$ is asymmetric, then for at least one value of $h$, \ref{case1.1} does not hold, and for at least one distinct value of $h$, \ref{case1.2} does not hold.
This implies that when $ \mathbf{X}(\theta_{t,\ell})$ is asymmetric, there are at least two values of $h$ such that $F_t^h$ does not contains two copies of $\theta_{t,\ell}$ that covers $e_1$ and $e_2$.
Therefore, 
\[
q\leq
\begin{cases}
t - (j-i)-1\leq t - 2, &\text{if } \mathbf{X}(\theta_{t,\ell}) \text{ is symmetric},\\
2(t - (j-i)-1)-2\leq 2(t - 3), &\text{if } \mathbf{X}(\theta_{t,\ell}) \text{ is asymmetric.}
\end{cases}
\]

\item $e_1\not\in E_{u_a}$ and $e_2\not\in E_{u_b}$ for every $a,b\in [s]$.\\
Let $e_1=v_iv_{i+1}$ and $e_2=v_jv_{j+1}$ (indices modulo $d$), where $i<j$. 
If \( j - i \ge t-2 \) and \( (i - j) \bmod d \ge t-2 \), then $q=0$. 
Since $d\geq 2t-3$, it is impossible two inequalities \( j - i < t-2 \) and \( (i - j) \bmod d < t-2 \) to hold simultaneously; thus, without loss of generality, we assume $j-i< t-2$.
Now, there are exactly $s(t-(j-i)-2)$ copies of $F_t$ containing both $e_1$ and $e_2$. If $\mathbf{X}(\theta_{t,\ell})$ is symmetric or asymmetric, then each $F_t^h$ contains one or two copies of $\theta_{t,\ell}$, respectively.
Therefore, 
\[
q\leq
\begin{cases}
s(t-(j-i)-2)\leq s(t - 3), &\text{if } \mathbf{X}(\theta_{t,\ell}) \text{ is symmetric},\\
2s(t-(j-i)-2)\leq 2s(t - 3), &\text{if } \mathbf{X}(\theta_{t,\ell}) \text{ is asymmetric.}
\end{cases}
\]

\item \label{caseA3} $e_1\in E_{u_a}$ for some $a\in [s]$ and $e_2\not\in E_{u_b}$ for every $b\in [s]$.\\
Let $e_1=u_av_i$ for some $a,i$ and $e_2=v_jv_{j+1}$  (indices modulo $d$) for some $j$. Note that $i=j$ is possible. 
%Let $d(i,j)$ denote the minimal cycle distance between vertices $v_i$ and $v_j$ on the cycle $C_d$.
If \( j - i \ge t-2 \) and \( (i - j) \bmod d \ge t-2 \), then $q=0$. 
Since $d\geq 2t-3$, it is impossible two inequalities \( j - i < t-2 \) and \( (i - j) \bmod d < t-2 \) to hold simultaneously; thus, without loss of generality, we assume $j-i< t-2$.
Now, there are exactly $m:=t-(j-i)-2$ copies of $F_t$ containing both $e_1$ and $e_2$. If $\mathbf{X}(\theta_{t,\ell})$ is symmetric or asymmetric, then each $F_t^h$ contains one or two copies of $\theta_{t,\ell}$, respectively.
Therefore, if $j\neq i$, then
\[
q\leq
\begin{cases}
t-(j-i)-2\leq t - 3, &\text{if } \mathbf{X}(\theta_{t,\ell}) \text{ is symmetric},\\
2(t-(j-i)-2)\leq 2(t - 3), &\text{if } \mathbf{X}(\theta_{t,\ell}) \text{ is asymmetric.}
\end{cases}
\]

If $j=i$, then $m=t-2$. 
Let $F_t^1=[u_a;v_i\cdots v_{i+t-2}]$,
$F_t^2=[u_a;v_{i-1}\cdots v_{i+t-3}]$,
$\ldots$, and $F_t^{m}=[u_a;v_{i+1-h} \ldots v_{i+1}]$ be the $m$ copies of $F_t$ containing both $e_1$ and $e_2$.
If $F_t^h$ contains a copy of $\theta_{t,\ell}$ that covers $e_1$ and $e_2$, then 
\begin{enumerate}[label={(\roman*)}]
    \item \label{case4.1} $h+1$ is a component of vector $\mathbf{X}(\theta_{t,\ell})$, or
    \item \label{case4.2} $t-h+1$ is a component of vector $\mathbf{X}(\theta_{t,\ell})$.
\end{enumerate}
Therefore, if \ref{case4.1}  or \ref{case4.2} holds for every $F_t^h$, then $3,\ldots,t-1$ are the components of vector $\mathbf{X}(\theta_{t,\ell})$, implying $\theta_{t,\ell}=F_t$. 
Thus, if $\theta_{t,\ell} \not= F_t$ and $\theta_{t,\ell}$ is symmetric, then $h+1$ being a component of vector $\mathbf{X}(\theta_{t,\ell})$ implies $t-h+1$ being a component of vector $\mathbf{X}(\theta_{t,\ell})$ and vice versa, i.e, \ref{case4.1} and \ref{case4.2} are equivalent. Consequently, there is at least one value of $h$ such that $F_t^h$ does not contains two copies of $\theta_{t,\ell}$ that covers $e_1$ and $e_2$ in such a situation. 
If $\theta_{t,\ell}$ is asymmetric, then for at least one value of $h$, \ref{case4.1} does not hold, and for at least one distinct value of $h$, \ref{case4.2} does not hold.
This implies that when $ \mathbf{X}(\theta_{t,\ell})$ is asymmetric, there are at least two values of $h$ such that $F_t^h$ does not contains two copies of $\theta_{t,\ell}$ that covers $e_1$ and $e_2$.
Therefore, 
\[
q\leq
\begin{cases}
t-2, &\text{if } \theta_{t,\ell} = F_t,\\
(t-2)-1= t - 3, &\text{if } \theta_{t,\ell} \not= F_t \text{ and } \mathbf{X}(\theta_{t,\ell}) \text{ is symmetric},\\
2(t-2)-2= 2(t - 3), &\text{if } \mathbf{X}(\theta_{t,\ell}) \text{ is asymmetric.}
\end{cases}
\]
\end{enumerate}

This completes the proof.
\end{proof}

\begin{lemma} \label{lem:p3}
Let $d,s,t,\ell$ be nonnegative integers such that $d\geq 3t-5$ and $t\geq \max\{4,\ell+3\}$ .
For any three edges in the graph $W_d(s)$, there exist at most $N_{\ref{lem:p3}}$ copies of the graph $\theta_{t,\ell}$ that contain at least two of these edges and exactly one hub vertex of $W_d (s)$, where
\[
N_{\ref{lem:p3}}=
\begin{cases}
\max\{s(t - 2),t-1\},  &\text{if } \mathbf{X}(\theta_{t,\ell}) \text{ is symmetric,}\\
\max\{2s(t-2),2(t-1)\}, &\text{if } \mathbf{X}(\theta_{t,\ell}) \text{ is asymmetric.}
\end{cases}
\]
\end{lemma}

\begin{proof}
Let $e_1$, $e_2$ and $e_3$ be three fixed edges in the graph $W_d(s)$. We define $q$ as the quantity representing the number of copies of the graph $\theta_{t,\ell}$ that  contain at least two of these three edges and exactly one hub vertex of $W_d (s)$.
The distribution of these three edges must fall into exactly one of the following mutually exclusive cases:
\begin{enumerate}[label={Case B\arabic*}]
\item $e_1\in E_{u_a}$, $e_2\in E_{u_b}$, and $e_3\in E_{u_c}$ for some $a,b,c\in [s]$.\\
If $a,b,c$ are pairwise distinct, then it is clear $q=0$. Thus, we assume $a=b$. If $c\neq a$, then
$q$ is the number of copies of the graph $\theta_{t,\ell}$ that  contain both $e_1$ and $e_2$, and exactly one hub vertex of $W_d(s)$. By the same analysis as in \ref{caseA1} of Lemma \ref{lem:p2}, we conclude
\[
q\leq
\begin{cases}
t - 2, &\text{if } \mathbf{X}(\theta_{t,\ell}) \text{ is symmetric},\\
2(t - 3), &\text{if } \mathbf{X}(\theta_{t,\ell}) \text{ is asymmetric.}
\end{cases}
\]
Thus, we further assume $c=a$. Let $e_1=u_av_i$, $e_2=u_av_j$, and $e_3=u_av_k$ ($i<j<k$). 

Since  \( d \ge 3t-5> 3(t-2) \), it is impossible all three inequalities \( j - i \leq t-2 \), \( k - j \leq t-2 \), and \( (i - k) \bmod d \leq t-2 \) to hold simultaneously; thus, without loss of generality, we assume \( (i - k) \bmod d > t-2 \). This implies that if $W_{d + 1}(s)$ contains a copy of the graph $F_t$ that  covers both $e_1$ and $e_3$, then it must cover $e_2$.
Thus, there are at most $(t+j-2)-j+1=t-1$ copies of $F_t$ containing both $e_1$ and $e_2$ or both $e_2$ and $e_3$. If $\mathbf{X}(\theta_{t,\ell})$ is symmetric or asymmetric, then each copy of $F_t$ contains one or two copies of $\theta_{t,\ell}$, respectively. Therefore, 
\[
q\leq
\begin{cases}
t - 1, &\text{if } \mathbf{X}(\theta_{t,\ell}) \text{ is symmetric},\\
2(t - 1), &\text{if } \mathbf{X}(\theta_{t,\ell}) \text{ is asymmetric.}
\end{cases}
\]

\item $e_1\not\in E_{u_a}$ , $e_2\not\in E_{u_b}$, and $e_3\not\in E_{u_c}$ for every $a,b,c\in [s]$.\\
Let $e_1=v_iv_{i+1}$ , $e_2=v_jv_{j+1}$ and $e_3=v_kv_{k+1}$ (indices modulo $d$), where $i<j<k$. 

Since \( d \ge 3t-5> 3(t-3) \), it is impossible for all three inequalities \( j - i \leq t-3 \), \( k - j \leq t-3 \), and \( (i - k) \bmod d \leq t-3 \) to hold simultaneously; thus, without loss of generality, we assume \( (i - k) \bmod d > t-3 \). This implies that if $W_{d + 1}(s)$ contains a copy of the graph $F_t$ that  covers both $e_1$ and $e_3$, then it must cover $e_2$, yielding at most $s((t+j-2)-(j+1)+1)=s(t-2)$ copies of $F_t$ covering both $e_1$ and $e_2$ or both $e_2$ and $e_3$. If $\mathbf{X}(\theta_{t,\ell})$ is symmetric or asymmetric, then each copy of $F_t$ contains one or two copies of $\theta_{t,\ell}$, respectively. Therefore, 
\[
q\leq
\begin{cases}
s((t+j-2)-(j+1)+1)\leq s(t - 2), &\text{if } \mathbf{X}(\theta_{t,\ell}) \text{ is symmetric},\\
2s((t+j-2)-(j+1)+1)\leq 2s(t - 2), &\text{if } \mathbf{X}(\theta_{t,\ell}) \text{ is asymmetric.}
\end{cases}
\]

\item \label{caseB3} $e_1,e_2\in E_{u_a}$ for some $a\in [s]$ and $e_3\not\in E_{u_b}$ for every $b\in[s]$.\\
Let $e_1=u_av_i$, $e_2=u_av_j$ with $i<j$ and $e_3=v_kv_{k+1}$  (indices modulo $d$) for some $k$. Note that $k=i$ or $k=j$ is possible. 

Suppose first $i\le k< j$. 
Since $d\ge 3t-5> 3t-7$, it is impossible for all three inequalities \(k - i \leq t-2 \), \( j - k \leq t-2 \), and \( (i - j) \bmod d \leq t-2 \) to hold simultaneously; thus we consider three subcases. 
\begin{itemize}
    \item If \(k - i > t-2\), then any copy of \(F_t\) in \(W_{d}(s)\) that covers both \(e_1\) and \(e_3\) must necessarily cover \(e_2\). Thus, there are at most $(t+j-2)-j+1=t-1$ copies of $F_t$ covering both $e_2$ and $e_3$ or both $e_1$ and $e_2$. 
    \item If \(j - k > t-2\), then any copy of \(F_t\) in \(W_{d}(s)\) that covers both \(e_2\) and \(e_3\) must necessarily cover \(e_1\). Thus, there are at most $(t+i-2)-i+1=t-1$ copies of $F_t$ covering both $e_1$ and $e_3$ or both $e_1$ and $e_2$. 
     \item If \( (i - j) \bmod d > t-2\), then any copy of \(F_t\) in \(W_{d}(s)\) that covers both \(e_1\) and \(e_2\) must necessarily cover \(e_3\). Thus, there are at most $(t+k-2)-(k+1)+1=t-2$ copies of $F_t$ covering both $e_1$ and $e_3$ or both $e_2$ and $e_3$. 
\end{itemize}

The case where either $k \geq j$ or $k < i$ can be handled in a similar manner, leading to the same conclusion that there exist at most $t - 1$ copies of $F_t$ that cover at least two of the three edges $e_1$, $e_2$, and $e_3$.
Because each $F_t^h$ contains one or two copies of $\theta_{t,\ell}$ when $\mathbf{X}(\theta_{t,\ell})$ is symmetric or asymmetric, respectively, we have
\[
q\leq
\begin{cases}
t - 1, &\text{if } \mathbf{X}(\theta_{t,\ell}) \text{ is symmetric},\\
2(t - 1), &\text{if } \mathbf{X}(\theta_{t,\ell}) \text{ is asymmetric.}
\end{cases}
\]

\item $e_1\in E_{u_a}$ for some $a\in [s]$, and $e_2\not\in E_{u_b}$, $e_3\not\in E_{u_c}$ for every $b,c\in [s]$.\\
Let $e_1=u_av_i$ , $e_2=v_jv_{j+1}$ and $e_3=v_kv_{k+1}$   (indices modulo $d$) for some $k,j$, where $j<k$.

Analogous to the approach in \ref{caseB3}, we partition the analysis into three primary cases based on the ordering of \(i\), \(j\), and \(k\): \(j \leq i < k\), \(i \geq k\), or \(i < j\), with each further subdivided into three subcases (e.g., for \(j \leq i < k\), these are \(i - j > t - 2\), \(k - i > t - 2\), and \((j - k) \bmod d > t - 2\)); across all nine subcases, we uniformly conclude that there exist at most \(t - 1\) copies of \(F_t\) covering at least two of the edges \(e_1\), \(e_2\), and \(e_3\). Because each $F_t^h$ contains one or two copies of $\theta_{t,\ell}$ when $\mathbf{X}(\theta_{t,\ell})$ is symmetric or asymmetric, respectively, we have
\[
q \leq
\begin{cases}
t - 1, & \text{if } \mathbf{X}(\theta_{t,\ell}) \text{ is symmetric,} \\
2(t - 1), & \text{if } \mathbf{X}(\theta_{t,\ell}) \text{ is asymmetric.}
\end{cases}
\]

\item $e_1\in E_{u_a}$, $e_2\in E_{u_b}$ for distinct $a,b\in [s]$, and $e_3\not\in E_{u_c}$ for every $c\in [s]$.\\
Applying the same proof technique as in \ref{caseA3} of Lemma~\ref{lem:p2} twice, first for the edge pair $\{e_1, e_2\}$ and second for $\{e_1, e_3\}$, we conclude
\[
q\leq
\begin{cases}
2(t-2)\leq s(t-2), &\text{if } \theta_{t,\ell} = F_t,\\
2(t-3)\leq s(t-3), &\text{if } \theta_{t,\ell} \not= F_t \text{ and } \mathbf{X}(\theta_{t,\ell}) \text{ is symmetric},\\
4(t-3)\leq 2s(t-3), &\text{if } \mathbf{X}(\theta_{t,\ell}) \text{ is asymmetric.}
\end{cases}
\]
Note that $s\geq 2$ in this case.
\end{enumerate}

This completes the proof. 
\end{proof}

\begin{lemma} \label{lem:p4+}
Let $d,s,t,\ell$ be integers such that $d\geq t-1$ and $t\geq \ell+3$.
For any $i\geq 4$ edges in the graph $W_d(s)$, there exist at most $N_{\ref{lem:p4+}}$ copies of the graph $\theta_{t,\ell}$ that contain at least two of these edges and exactly one hub vertex of $W_d (s)$ , where
\[
N_{\ref{lem:p4+}}=
\begin{cases}
\frac{(t-1)i}{2}, &\text{if } \theta_{t,\ell} = F_t \text{ and } s=1,\\
s(t-2)i &\text{if } \mathbf{X}(\theta_{t,\ell}) \text{ is asymmetric,}\\
\frac{s(t-2)i}{2},  &\text{otherwise.} 
\end{cases}  
\]
\end{lemma}

\begin{proof}
Let $x $ and $ y $ denote the number of spokes and rim edges, respectively, among the given $ i $ edges (so $ x + y = i $). The following properties hold for $ W_d (s) $. If $ \mathbf{X}(\theta_{t,n})$ is symmetric, then each spoke appears in $ \ell + 2 $ distinct $ \theta_{t,\ell} $-subgraphs and each rim edge in $ s(t - 2) $; if asymmetric, then each spoke appears in $ 2(\ell + 2) $ and each rim edge in $ 2s(t - 2) $. We define \( q \) as the number of \( \theta_{t,\ell} \)-copies that contain at least two of the \( i \) edges and exactly one hub vertex from \( W_d(s) \). Additionally, we define \( q_j \) as the number of such \( \theta_{t,\ell} \)-copies that contain exactly \( j \) edges and exactly one hub vertex from \( W_d(s) \).
Now, the definition of $q$ and $q_j$ implies
\begin{align*}
    &2q 
    \le 2q + q_1 
    \le \sum_{j\ge 1} j q_j \\
    =& 
  \begin{cases}
(t-1)x+s(t-2)y \leq (t-1)(x+y)=(t-1)i, &\text{if } \theta_{t,\ell} = F_t \text{ and } s=1,\\
(t-1)x+s(t-2)y \leq s(t-2)(x+y)=s(t-2)i, &\text{if } \theta_{t,\ell} = F_t \text{ and } s\ge2,\\
(\ell+2)x+s(t-2)y \leq s(t-2)(x+y)=s(t-2)i,  &\text{if } \theta_{t,\ell} \not= F_t \text{ and } \mathbf{X}(\theta_{t,\ell}) \text{ is symmetric},\\
2(\ell+2)x+2s(t-2)y \leq 2s(t-2)(x+y)=2s(t-2)i, &\text{if } \mathbf{X}(\theta_{t,\ell}) \text{ is asymmetric.}
\end{cases}  
\end{align*}
This proves the lemma.
\end{proof}

\begin{lemma} \label{lem:lucky7}
Suppose that $ W_{d}(s) $  contains a copy of the graph $ F_t $.
If $s=2$ and $t\geq 6$, or $s\geq 3$ and $t\geq 7$, then this
$ F_t $ contains exactly one hub vertex of $ W_{d}(s) $.
 \end{lemma}

\begin{proof}
Suppose $ F_t $ is embedded in $ W_{d}(s) $. Assume for contradiction that $ F_t $ contains at least two hub vertices $ u_1$ and $ u_2 $ of $ W_{d}(s) $. Since hubs are independent (i.e., $ u_1u_2 \notin E(W_{d}(s)) $), the central vertex $ v $ of $ F_t $ cannot be a hub. Thus, $ v $ must lie on the boundary cycle $ C $ of $ W_{d}(s) $, where it has exactly two neighbors $ w_1 $ and $ w_2 $ on $ C $. The remaining $ t-3 $ neighbors of $ v $ (including $u_1$ and $u_2$) in $ F_t $ must then be hubs (as they cannot reside on $ C $ without violating the embedding). If $s=2$, then $ t-3 \leq 2$, implying $t\leq 5$, a contradiction. If $s\geq 3$, then these $ t-3 $ hubs must form an independent set in the boundary cycle of $ F_t $, implying $ t-3 \leq \lfloor t/2 \rfloor $ (since no two hubs can be adjacent in the cycle). This inequality simplifies to $ t \leq 6 $, contradicting $ t \geq 7 $.
Therefore, $ F_t $ contains at most one hub vertex of $ W_{d}(s) $. Given that the central vertex of $ F_t $ has degree $ t-1 \geq 5 $, while every vertex on $ C $ has degree at most $ 3 $ in $ F_t $, the central vertex must be a hub itself. This completes the proof.
\end{proof}

   Let $G$ be an edge-colored graph. For a positive integer $i$, we define $A_i(G)$ as the set of colors that appear on exactly $i$ edges of $G$. For any color $c$  and any subgraph $H\subseteq G$, we introduce the notation $p_j(c,G,H)$ to represent the number of $H$-subgraphs within the graph $G$ that contain exactly $j$ edges colored with the color $c$. The total number of $H$-subgraphs in $G$ that have at least two edges colored with the color $c$ is given by the following sum:
$$
p(c,G,H)=\sum_{j\geq2}p_j(c,G,H).
$$
% When the graph $G$ are clear from the context, we simplify the notations. That is, we use $p_j(c,H)$ and $p(c,H)$ in place of $p_j(c,H,G)$ and $p(c,H,G)$ respectively.
Applying Lemmas \ref{lem:p2}, \ref{lem:p3}, \ref{lem:p4+}, and \ref{lem:lucky7}, we conclude the following.

\begin{lemma}\label{lem:all-used-bounds}
Let $d,s,t,\ell$ be nonnegative integers such that $d\geq 3t-5$ and $t\geq \max\{4,\ell+3\}$.
Given an edge-colored $ W_{d}(s) $ and a color $ c $, we have
\begin{align*}
p(c, W_{d}, F_t) 
\le&
\begin{cases}
t-2, & \text{if } c \in A_2(W_{d}), \\
t-1, & \text{if } c \in A_3(W_{d}), \\
\frac{(t-1)i}{2}, & \text{if } c \in A_i(W_{d}) \text{ for } i \geq 4.
\end{cases}
\\
p(c, W_{d}, \theta_{t,\ell}) 
\le&
\begin{cases}
\beta(t-3), & \text{if } c \in A_2(W_{d}), \\
\beta(t-1), & \text{if } c \in A_3(W_{d}), \\
\frac{\beta(t-2)i}{2},  & \text{if } c \in A_i(W_{d}) \text{ for } i \geq 4,
\end{cases}\\
\text{ where }
\beta=& 
\begin{cases}
1, & \text{if } \theta_{t,\ell} \neq F_t \text{ and } \mathbf{X}(\theta_{t,\ell}) \text{ is symmetric}, \\
2 &\text{if } \mathbf{X}(\theta_{t,\ell}) \text{ is asymmetric,}\\ 
\end{cases}  \\
\text{ and if } t\geq 6 \text{ and } s= 2, 
\text{ or } t\geq 7 \text{ and } s\geq 3,
\text{ then}\\
    p(c, W_{d}(s), F_t) 
\le&
\begin{cases}
s(t-3), & \text{if } c \in A_2(W_{d}(s)), \\
s(t-2), & \text{if } c \in A_3(W_{d}(s)), \\
\frac{s(t-2)i}{2},  & \text{if } c \in A_i(W_{d}(s)) \text{ for } i \geq 4.
\end{cases}
\end{align*}
\end{lemma}

\begin{proof}
    There are a total of nine upper bounds from top to bottom, each derived based on Lemma \ref{lem:p2}, \ref{lem:p3}, \ref{lem:p4+}, \ref{lem:p2}, \ref{lem:p3}, \ref{lem:p4+}, \ref{lem:p2}, \ref{lem:p3} and \ref{lem:p4+}, respectively.
\end{proof}

% --------NOTE------------  

% in $ (t-1) $ distinct $ F_t $-subgraphs; 
% and each rim edge appears in $ 2(t-2) $ distinct $ F_t $-subgraphs. 

% ---------------------

% \begin{itemize}
% \item $\theta_{t,n}$ is not a symmetric graph.
% \end{itemize}

%  In this case, the following properties hold for $ W_{d} $: 
% the number of $ \theta_{t,n} $-subgraphs is exactly $ 2d $; 
% each spoke appears in $ 2(n+2) $ distinct $ F_t $-subgraphs; 
% and each rim edge appears in $ 2(t-2) $ distinct $ \theta_{t,n} $-subgraphs

% -----------------

% \begin{itemize}
% \item $\theta_{t,n}$ is a symmetric graph.
% \end{itemize}

%  In this case, the following properties hold for $ W_{d} $: 
% the number of $ \theta_{t,n} $-subgraphs is exactly $ d $; 
% each spoke appears in $ (n+2) $ distinct $ F_t $-subgraphs; 
% and each rim edge appears in $ (t-2) $ distinct $ \theta_{t,n} $-subgraphs. 

% ------------------

% The following properties hold for $ W^*_d $: 
% the number of $ F_t $-subgraphs is exactly $ 2d $; 
% each spoke appears in $ (t-1) $ distinct $ F_t $-subgraphs; 
% and each rim edge appears in $ 2(t-2) $ distinct $ F_t $-subgraphs. 

% ----------------
% Case1: $\theta_{t,n}$ is not a centrally symmetrical graph. In this case, $W^*_d$ contains a total of $4d$ $\theta_{t,n}$. Each spoke is in exactly $2(n+2)$ $\theta_{t,n}$ and each rim edge is in exactly $4(t-2)$ $\theta_{t,n}$.

% ------

% Case2: $\theta_{t,n}$ is a central symmetric graph. In this case, $W^*_d$ contains a total of $2d$ $\theta_{t,n}$. Each spoke is in exactly $(n+2)$  $\theta_{t,n}$ and each rim edge is in exactly $2(t-2)$  $\theta_{t,n}$.

\section{Optimal colorings with rainbow structures} \label{sec:upperbound}

\begin{thm} \label{thm:W-F}
Let $d$ and $t$ be integers such that $d\geq 3t-5 $ and $t\ge 4$.
Any edge coloring of $W_{d}$ with at least
$\left\lfloor \frac{2t-5}{t-2}d \right\rfloor + 1$ colors contains a rainbow subgraph isomorphic to $F_t$.
In other words,
$${\rm rb}(W_d ,F_t)\leq \left\lfloor \frac{2t-5}{t-2} d\right\rfloor  + 1.$$
\end{thm}

\begin{proof}
Let $\varphi$ be an edge $ k $-coloring of $ W_{d} $, where
$
k := \left\lfloor \frac{2t-5}{t-2}d \right\rfloor + 1.
$
Suppose, for the sake of contradiction, that no rainbow $ F_t $ subgraph exists in $ W_{d} $.
Thus, each $F_t$ contained in $W_{d}$ has a color $c$ with $p(c):=p(c,W_{d},F_t)\ge 1$.
Set $A_i:=A_i(W_{d})$.
Given that the number of $F_t$ contained in $W_{d}$ is exactly $d$ and $2d=|E(W_{d})|=\sum\limits_{i\ge 1}i|A_i|$, we conclude 
\begin{align}
    \notag d &\le \sum\limits_{c\in [k] }p(c) = \sum\limits_{i\ge 2}\sum\limits_{c\in A_i}p(c) = \sum\limits_{c\in A_2}p(c)+\sum\limits_{c\in A_3}p(c)+\sum\limits_{i\ge 4}\sum\limits_{c\in A_i}p(c)
    \\
   \notag  &\le (t-2)|A_2|+(t-1)|A_3|+\frac{(t-1)}{2} \sum\limits_{i\ge 4} i|A_i| \\
   \notag &= (t-2)|A_2|+(t-1)|A_3|+\frac{(t-1)}{2}(2d-|A_1|-2|A_2|-3|A_3|)\\
   \notag &= -\left( \frac{t-1}{2}|A_1|+|A_2|+\frac{t-1}{2}|A_3| \right)+(t-1)d.
\end{align}
Here, the second inequality follows from Lemma~\ref{lem:all-used-bounds}. It follows
\begin{align} \label{eq:1-1}
    \frac{t-1}{2}|A_1|+|A_2|+\frac{t-1}{2}|A_3| \le (t-2)d.
\end{align}

If $ t= 4$, then \eqref{eq:1-1} implies
\begin{align*}
    3|A_1|+2|A_2|+|A_3| \le 4d.   
\end{align*}
Thus,
\begin{align*}
    k=\left\lfloor \frac{3}{2}d \right\rfloor + 1 =\sum_{i\geq 1}|A_i| 
    &\le \frac{1}{4}(3|A_1|+2|A_2|+|A_3|) + \sum\limits_{i\ge1} \frac{i}{4}|A_i| \\
    &\le \frac{1}{4} \times 4d + \frac{1}{4} \times 2d =\frac{3}{2}d,
\end{align*}
a contradiction.

If $t\geq 5$, then applying $d\ge \frac{1}{2}\sum\limits_{i=1}^{t-3} i|A_i|$, we reformulate \eqref{eq:1-1} as follows:
\begin{align*}
    \frac{t-1}{2}|A_1|+|A_2|+\frac{t-1}{2}|A_3| \le (t-2)d &=  (2t-7)d-(t-5)d \\
         &\le (2t-7)d-\frac{t-5}{2}\sum\limits_{i=1}^{t-3} i|A_i|.  
\end{align*}
This implies
\begin{align*}
    (2t-7)d  &\geq (t-3)|A_1|+(t-4)|A_2|+(2t-8)|A_3|+\sum_{i=4}^{t-3}\frac{(t-5)i}{2}|A_i|\\
            &\geq (t-3)|A_1|+(t-4)|A_2|+(t-5)|A_3|+\sum_{i=4}^{t-3} (t-i-2)|A_i| 
             =\sum_{i=1}^{t-3} (t-i-2)|A_i|.
\end{align*}
Thus,
\begin{align*}
    k=\bigg\lfloor \frac{2t-5}{t-2}d  \bigg\rfloor +1 &=\sum_{i\geq 1}|A_i| 
    \le \frac{1}{t-2}\sum_{i=1}^{t-3} (t-i-2)|A_i| +\sum\limits_{i\ge 1} \frac{i}{t-2}|A_i| \\
    &\le \frac{1}{t-2} \times (2t-7)d + \frac{1}{t-2} \times 2d
    =\frac{2t-5}{t-2}d.
\end{align*}
This contradiction completes the proof.
\end{proof}

\begin{thm} \label{thm:Ws-F-and-Wd-theta}
Let $d$, $s$, $t$ and $\ell$ be integers such that $d\geq 3t  -5$ and $t\ge \max\{5, \ell+4\}$.
The following holds:
\begin{enumerate}
    \item any edge coloring of $W_{d}(s)$ with at least
$\left\lfloor\frac{(s+1)t-(3s+4)}{t-3}d \right\rfloor +1$ colors contains a rainbow $F_t$ if we further have $s=2$ and $t\geq 6$, or $s\geq 3$ and $t\geq 7;$
  \item any edge coloring of $W_{d}$ with at least
$\left\lfloor\frac{2t-7}{t-3}d\right\rfloor +1$ colors contains a rainbow $\theta_{t,\ell}$.
\end{enumerate}
In other words,
\begin{align*}
{\rm rb}(W_d (s),F_t) &\leq \bigg\lfloor\frac{(s+1)t-(3s+4)}{t-3}d\bigg\rfloor +1~(s=2 {\rm ~and }~t\geq 6, {\rm ~or }~ s\geq 3 {\rm ~and }~ t\geq 7);\\
{\rm rb}(W_d,\theta_{t,\ell}) &\leq \bigg\lfloor\frac{2t-7}{t-3}d\bigg\rfloor +1.
\end{align*}
\end{thm}

\begin{proof}
Let $ \varphi $ be an edge $ k $-coloring of $ W_{d}(s) $, where
$k := \lfloor\frac{(s+1)t-(3s+4)}{t-3}d\rfloor +1.$
For convenience, we define \( H \) as follows:
\[
H := 
\begin{cases} 
F_t & \text{if } s=2 {\rm ~and }~t\geq 6, {\rm ~or }~ s\geq 3 {\rm ~and }~ t\geq 7, \\
\theta_{t,\ell} & \text{if } s = 1.
\end{cases}
\]
Lemma \ref{lem:lucky7} implies that each $H$ contained in $ W_d (s) $ contains exactly one hub vertex of $ W_{d}(s) $.

Suppose, for the sake of contradiction, that no rainbow $H$ exists in $ W_d (s) $.
Thus, each $H$ contained in $W_{d}(s)$ has a color $c$ with $p(c):=p(c,W_{d}(s),H)\ge 1$.
Set $A_i:=A_i(W_{d}(s))$. The number of $ H $ contained in $ W_{d}(s) $ is exactly $ sd $ or $ 2sd $, depending on whether $ \mathbf{X}(H) $ is symmetric or asymmetric, respectively, and $$(s+1)d=|E(W_{d}(s))|=\sum\limits_{i\ge 1}i|A_i|.$$ 
We consider the following two cases.

\begin{enumerate}[label={Case \arabic*}]

\item $s\geq 2$.

By Lemma \ref{lem:all-used-bounds}, if $ \mathbf{X}(H) $ is symmetric, then
\begin{align}
    \notag sd &\le \sum\limits_{c\in [k] }p(c) = \sum\limits_{i\ge 2}\sum\limits_{c\in A_i}p(c) = \sum\limits_{c\in A_2}p(c)+\sum\limits_{c\in A_3}p(c)+\sum\limits_{i\ge 4}\sum\limits_{c\in A_i}p(c)
    \\
   \notag  &\le s(t-3)|A_2|+s(t-2)|A_3|+ \frac{s(t-2)}{2}\sum\limits_{i\ge 4} i|A_i| \\
   \notag &= \frac{s}{2} \bigg( 2(t-3)|A_2|+2(t-2)|A_3|+ (t-2) \big((s+1)d-|A_1|-2|A_2|-3|A_3|\big) \bigg)\\
   \notag &= \frac{s}{2} \bigg( (s+1)(t-2)d-(t-2)|A_1|-2|A_2|-(t-2)|A_3| \bigg),
\end{align}
and if $ \mathbf{X}(H) $ is asymmetric, then
\begin{align}
    \notag 2sd &\le \sum\limits_{c\in [k] }p(c) = \sum\limits_{i\ge 2}\sum\limits_{c\in A_i}p(c) = \sum\limits_{c\in A_2}p(c)+\sum\limits_{c\in A_3}p(c)+\sum\limits_{i\ge 4}\sum\limits_{c\in A_i}p(c)
    \\
   \notag  &\le 2s(t-3)|A_2|+2s(t-2)|A_3|+s(t-2)\sum\limits_{i\ge 4} i|A_i| \\
   \notag &= s\bigg(2(t-3)|A_2|+2(t-2)|A_3|+(t-2)((s+1)d-|A_1|-2|A_2|-3|A_3|)\bigg)\\
   \notag &= s\bigg((s+1)(t-2)d  -(t-2)|A_1|-2|A_2|-(t-2)|A_3| \bigg).
\end{align}
In either case, we have
\begin{align} \label{eq:1-2}
    (t-2)|A_1|+2|A_2|+(t-2)|A_3| \le \bigg((s+1)(t-2)-2\bigg)d.
\end{align}

If $t=5$, then \eqref{eq:1-2} implies
\begin{align*}
    3|A_1|+2|A_2|+|A_3|\le (3s+1)d .
\end{align*}
Thus,
\begin{align*}
    k=\left \lfloor \frac{2s+1}{2}d \right \rfloor +1 
    =\sum_{i\ge1} |A_i|
    &\le \frac{1}{4}(3|A_1|+2|A_2|+|A_3|) + 
    \sum \limits_{i\ge 1} \frac{i}{4}|A_i| \\
    &\le \frac{1}{4} \times (3s+1)d +\frac{1}{4} \times (s+1)d=\frac{2s+1}{2} d,
\end{align*} 
a contradiction.

If $t\ge 6$, then applying $d\ge \frac{1}{s+1}\sum \limits_{i=1}^{t-4}i|A_i|$, we reformulate \eqref{eq:1-2} as follows:
\begin{align*}
    (t-2)|A_1|+2|A_2|+(t-2)|A_3|&\le \bigg((s+1)(t-2)-2\bigg)d \\&=\bigg((2s+2)t-(8s+10)\bigg)d-\bigg((s+1)t-(6s+6)\bigg)d \\
    &\le \bigg((2s+2)t-(8s+10)\bigg)d - (t-6)\sum\limits_{i=1}^{t-4}i|A_i|.
\end{align*}
This implies
\begin{align*}
    \bigg((2s+2)t-(8s+10)\bigg)d \ge (2t-8)|A_1|+(2t-10)|A_2|+(4t-20)|A_3|+(t-6)\sum\limits_{i=4}^{t-4}i|A_i|.
\end{align*}
Since $t\geq 6$ and  $(t-6)i\geq 2(t-3-i)$ for $i\geq 4$, we further have
\begin{align*}
   \bigg ((s+1)t-(4s+5)\bigg)d &\ge (t-4)|A_1|+(t-5)|A_2|+(t-6)|A_3|+\sum\limits_{i=4}^{t-4}(t-3-i)|A_i|\\
              & =\sum \limits_{i=1}^{t-4}(t-3-i)|A_i|.
\end{align*}
Thus,
\begin{align*}
    k &= \bigg \lfloor \frac{(s+1)t-(3s+4)}{t-3} d \bigg \rfloor +1 =\sum_{i\ge 1}|A_i|   \\
    &\le \frac{1}{t-3}\sum_{i=1}^{t-4} (t-3-i)|A_i| +\sum\limits_{i\ge 1}\frac{i}{t-3}|A_i| \\
    &\le \frac{1}{t-3} \times \bigg((s+1)t-(4s+5) \bigg)d +\frac{1}{t-3} \times (s+1)d
    =\frac{(s+1)t-(3s+4)}{t-3}d,
\end{align*}
a contradiction.

\item $s=1$.

By Lemma \ref{lem:all-used-bounds}, if $ \mathbf{X}(H) $ is symmetric (in this situation we have $H\neq F_t$ since $t\neq \ell + 3$), then
\begin{align}
    \notag d &\le \sum\limits_{c\in [k] }p(c) = \sum\limits_{i\ge 2}\sum\limits_{c\in A_i}p(c) = \sum\limits_{c\in A_2}p(c)+\sum\limits_{c\in A_3}p(c)+\sum\limits_{i\ge 4}\sum\limits_{c\in A_i}p(c)
    \\
   \notag  &\le (t-3)|A_2|+(t-1)|A_3|+\frac{t-2}{2}\sum\limits_{i\ge 4} i|A_i| \\
   \notag &= (t-3)|A_2|+(t-1)|A_3|+\frac{t-2}{2}(2d-|A_1|-2|A_2|-3|A_3|)\\
   \notag &= -\bigg(\frac{t-2}{2}|A_1|+|A_2|+\frac{t-4}{2}|A_3| \bigg)+(t-2)d,
\end{align}
and if $ \mathbf{X}(H) $ is asymmetric, then
\begin{align}
    \notag 2d &\le \sum\limits_{c\in [k] }p(c) = \sum\limits_{i\ge 2}\sum\limits_{c\in A_i}p(c) = \sum\limits_{c\in A_2}p(c)+\sum\limits_{c\in A_3}p(c)+\sum\limits_{i\ge 4}\sum\limits_{c\in A_i}p(c)
    \\
   \notag  &\le 2(t-3)|A_2|+2(t-1)|A_3|+(t-2)\sum\limits_{i\ge 4} i|A_i| \\
   \notag &= 2(t-3)|A_2|+2(t-1)|A_3|+(t-2)(2d-|A_1|-2|A_2|-3|A_3|)\\
   \notag &= -\bigg((t-2)|A_1|+2|A_2|+(t-4)|A_3| \bigg)+2(t-2)d.
\end{align}
In either case, we have 
\begin{align} \label{eq:1-3}
    (t-2)|A_1|+2|A_2|+(t-4)|A_3| \le 2(t-3)d.
\end{align}

If $t=5$, then \eqref{eq:1-3} implies 
\begin{align*}
    3|A_1|+2|A_2|+|A_3| \le 4d,
\end{align*}
Thus,
\begin{align*}
    k =\left\lfloor \frac{3}{2}d \right\rfloor + 1 =\sum_{i\geq 1}|A_i| &\le \frac{1}{4}(3|A_1|+2|A_2|+|A_3|) + \sum\limits_{i\ge1} \frac{i}{4}|A_i| \\
    &\le \frac{1}{4} \times 4d + \frac{1}{4} \times 2d =\frac{3}{2}d,
\end{align*}
a contradiction.

If $t\geq 6$, then applying $d\ge \frac{1}{2}\sum\limits_{i=1}^{t-4} i|A_i|$, we reformulate \eqref{eq:1-3} as follows:
\begin{align*}
    (t-2)|A_1|+2|A_2|+(t-4)|A_3| \le (2t-6)d &=  (4t-18)d-(2t-12)d \\
         &\le (4t-18)d-(t-6)\sum\limits_{i=1}^{t-4} i|A_i|.  
\end{align*}
This implies
\begin{align*}
    (4t-18)d  &\geq (2t-8)|A_1|+(2t-10)|A_2|+(4t-22)|A_3|+\sum_{i=4}^{t-4} (t-6)i|A_i|.
\end{align*}
Since $t\geq 6$ and $(t-6)i\ge t-3-i$ for $i\geq 4$, we further have
\begin{align*}
    (2t-9)d  &\geq (t-4)|A_1|+(t-5)|A_2|+(t-6)|A_3|+\sum_{i=4}^{t-4}(t-3-i)|A_i|   \\        
            & =\sum\limits_{i=1}^{t-4} (t-3-i)|A_i|.
\end{align*}
Thus,
\begin{align*}
    k =\bigg\lfloor \frac{2t-7}{t-3}d  \bigg\rfloor +1 =\sum_{i\geq 1}|A_i| 
    &\le \frac{1}{t-3}\sum_{i=1}^{t-4} (t-i-3)|A_i| +\sum\limits_{i\ge 1} \frac{i}{t-3}|A_i| \\
    &\le \frac{1}{t-3} \times (2t-9)d + \frac{1}{t-3} \times 2d =\frac{2t-7}{t-3}d,
\end{align*}
a contradiction.
\end{enumerate}

In each of these cases, we arrive at contradictions, thereby concluding the proof.
\end{proof}

\section{Extremal graphs} \label{sec:lowerbound}

\begin{thm} \label{thm:extremal-1}
For integers $d,s,t$ such that $d\ge t-1$ and $t\geq 4$, we have
    $${\rm rb}(W_d ,F_t)\geq \left\lfloor \frac{2t-5}{t-2}d \right\rfloor +1.$$
\end{thm}

\begin{proof}
    For $W_d $ with hub $u$ and boundary vertices $v_1, v_2, \cdots, v_d$, we partition the spokes $uv_1, uv_2, \cdots, uv_d$ into $q$ groups starting from $uv_1$, with each group containing consecutive $t - 2$ spokes (the last group may have fewer than $t - 2$ spokes).

For its $m$-th group ($1\leq m\leq q - 1$) of spokes 
$uv_{(m - 1)(t - 2)+1},uv_{(m - 1)(t - 2)+2},\cdots,uv_{m(t - 2)}$,
color the first two spokes $uv_{(m - 1)(t - 2)+1}$ and $uv_{(m - 1)(t - 2)+2}$ with color $d+(m - 1)(t - 3)+1$, and for $3\leq i \leq t - 2$, color $uv_{(m - 1)(t - 2)+i}$ with color $d+(m - 1)(t - 3)+(i - 1)$. Thus, the last spoke $uv_{m(t - 2)}$ of the $m$-th group is colored with $d+m(t - 3)$. For the last group (the $q$-th group):
\begin{itemize}
    \item if the structure contains at least two spokes, then color the first two spokes \( u v_{(q - 1)(t - 2)+1} \) and \( uv_{(q - 1)(t - 2)+2} \) with color \( d + (q - 1)(t - 3) + 1 \). For the remaining spokes, assign colors by incrementing the previous spoke's color by one;
    \item if the structure contains exactly one spoke $uv_d$, then color it with the same color as $uv_{d - 1}$, which is $d+(q - 1)(t - 3)$.
\end{itemize}
 The purpose of these rules is to ensure that among any $t - 1$ consecutive spokes, there are at least two spokes with the same color.
 Therefore, while coloring the edges of each rim edge $v_iv_{i + 1}$ of $W_d $ with color $i$, we complete an edge coloring of $W_d $ avoiding the rainbow $F_t$.
 Since $d=(q-1)(t-2)+p$ where $1\leq p\leq t-2$, and this coloring uses 
 \begin{align*}
   & d+  (q-1)(t-3)+\left(d-(q-1)(t-2)-1\right)\\=&2d-sq
   = 
\begin{cases} 
2d-\left\lfloor \frac{d}{t - 2} \right\rfloor+1= \left\lfloor \frac{2t-5}{t-2}d \right\rfloor, & \text{if } p < t - 2, \\
2d-\frac{d}{t - 2}  = \frac{2t-5}{t-2}d, & \text{if } p = t - 2
\end{cases} 
\end{align*}
colors, we conclude the result of this theorem. 
\end{proof}

\begin{thm} \label{thm:extremal-2}
For integers $d,t,s,\ell$ such that $d\ge t-1$ and $t\ge \max\{5,\ell+4\}$, we have
\begin{align*}
     {\rm rb}(W_d ,\theta_{t,\ell})\geq \bigg\lfloor\frac{2t-7}{t-3}d\bigg\rfloor +1,
\end{align*}
and furthermore, if $s=2$ and $t\geq 6$, or $s\geq 3$ and $t\geq 7$, then
\begin{align*}
     {\rm rb}(W_d(s),F_t) \geq \bigg\lfloor\frac{(s+1)t-(3s+4)}{t-3}d\bigg\rfloor +1.
\end{align*}
\end{thm}

\begin{proof}
    For $W_d (s)$ with hub vertices $u_1,u_2,\ldots,u_s$ and boundary vertices $v_1, v_2, \cdots, v_d$, we partition the rim edges $v_1v_2, v_2v_3, \cdots, v_{d-1}v_d, v_dv_1$ into $q$ groups starting from $v_1v_2$, with each group containing consecutive $t - 3$ rim edges (the last group may have fewer than $t - 3$ rim edges). 
 
For the $m$-th group ($1\leq m\leq q - 1$), where the rim edges are $v_{(m - 1)(t - 3)+1}v_{(m - 1)(t - 3)+2}$,
$v_{(m - 1)(t - 3)+2} \allowbreak v_{(m - 1)(t - 3)+3},\cdots,v_{m(t - 3)}v_{m(t-3)+1}$, color the first two rim edges $v_{(m - 1)(t - 3)+1}v_{(m - 1)(t - 3)+2}$ and $v_{(m - 1)(t - 3)+2} \allowbreak v_{(m - 1)(t - 3)+3}$ with color $sd+(m - 1)(t - 4)+1$, and for $3\leq i \leq t - 3$, color $v_{(m - 1)(t - 3)+i}v_{(m-1)(t-3)+i+1}$ with color $sd+(m - 1)(t - 4)+(i - 1)$. Thus, the last rim edge $v_{m(t - 3)}v_{m(t-3)+1}$ of the $m$-th group is colored with $sd+m(t - 4)$. For the last group (the $q$-th group):
\begin{itemize}
    \item if it contains at least two rim edges, then color the first two rim edges $v_{(q - 1)(t - 3)+1}v_{(q - 1)(t - 3)+2}$ and $v_{(q - 1)(t - 3)+2}v_{(q - 1)(t - 3)+3}$ with color $sd+(q - 1)(t - 4)+1$, and the remaining rim edges with colors that are incremented by 1 one-by-one compared to the previous rim edge in the group;
    \item if it contains exactly one rim edge $v_dv_1$, then color it with the same color as $v_{d - 1}v_d$, which is $sd+(q - 1)(t - 4)$.
\end{itemize}
 The purpose of these rules is to ensure that among any $t - 2$ consecutive rim edges, there are at least two rim edges with the same color.
 Therefore, while coloring the edges of each spoke $u_av_i$ of $W_{d}(s)$ with color $(a-1)d+i$ for each $a\in [s]$, we complete a coloring of $W_d(s)$ such that
 \begin{itemize}
     \item there is no rainbow $F_t$ if $s=2$ and $t\geq 6$, or $s\geq 3$ and $t\geq 7$, and furthermore;
     \item there is no rainbow $\theta_{t,\ell}$ if $s=1$.
 \end{itemize}
 Since $d=(q-1)(t-3)+p$ where $1\leq p\leq t-3$, and this coloring uses
 \begin{align*}
      & sd+ (q-1)(t-4)+\left(d-(q-1)(t-3)-1\right)\\
      =& (s+1)d-q=  
\begin{cases} 
(s+1)d-\left\lfloor \frac{d}{t - 3} \right\rfloor-1= \left\lfloor \frac{(s+1)t-(3s+4)}{t-3}d \right\rfloor, & \text{if } p < t - 3, \\
(s+1)d-\left\lfloor \frac{d}{t - 3} \right\rfloor= \frac{(s+1)t-(3s+4)}{t-3}d, & \text{if } p = t - 3
\end{cases}
 \end{align*}
colors, we conclude the result of the theorem.
\end{proof}

In Figure~\ref{fig:extremal}(a), we present a 13-edge-coloring of $W_8$ following the construction in Theorem~\ref{thm:extremal-1}, ensuring no rainbow $F_5$ subgraph exists. Similarly, in Figure~\ref{fig:extremal}(b), we apply the proof technique in Theorem~\ref{thm:extremal-2} to construct a 21-edge-coloring of $W_8(2)$ that avoids rainbow $F_6$ subgraphs. However, this coloring inadvertently contains a rainbow $F_5$ (highlighted in gray), implying that this is not an extremal graph for $\mathrm{rb}(W_8(2),F_5)$. Actually, as shown in Section~\ref{sec:final}, we will prove $\mathrm{rb}(W_8(2),F_5) \le 19$. Thus, any 21-edge-coloring of $W_8(2)$ must contain a rainbow $F_5$ subgraph.

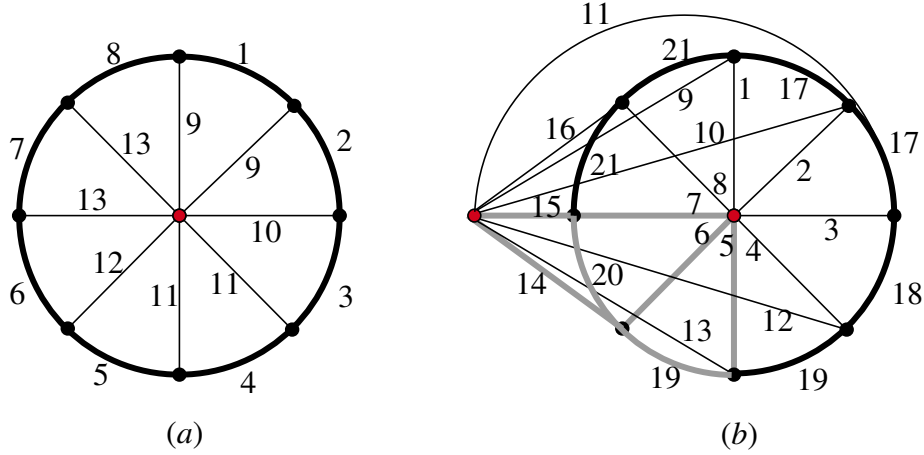
\begin{figure}[h]
    \centering

\tikzset{every picture/.style={line width=0.6pt}} %set default line width to 0.75pt        

\begin{tikzpicture}[x=0.6pt,y=0.6pt,yscale=-1,xscale=1]
%uncomment if require: \path (0,734); %set diagram left start at 0, and has height of 734

%Shape: Circle [id:dp6631949804534567] 
\draw  [color={rgb, 255:red, 0; green, 0; blue, 0 }  ,draw opacity=1 ][line width=2.25]  (30.67,311) .. controls (30.67,255.77) and (75.44,211) .. (130.67,211) .. controls (185.9,211) and (230.67,255.77) .. (230.67,311) .. controls (230.67,366.23) and (185.9,411) .. (130.67,411) .. controls (75.44,411) and (30.67,366.23) .. (30.67,311) -- cycle ;
%Straight Lines [id:da8838882000125192] 
\draw    (130.67,311) -- (230.67,311) ;
%Flowchart: Connector [id:dp4126932780040673] 
\draw  [fill={rgb, 255:red, 0; green, 0; blue, 0 }  ,fill opacity=1 ] (126.67,311) .. controls (126.67,308.79) and (128.46,307) .. (130.67,307) .. controls (132.88,307) and (134.67,308.79) .. (134.67,311) .. controls (134.67,313.21) and (132.88,315) .. (130.67,315) .. controls (128.46,315) and (126.67,313.21) .. (126.67,311) -- cycle ;
%Straight Lines [id:da35047672053162615] 
\draw    (130.67,311) -- (130.67,411) ;
%Straight Lines [id:da8777349966944026] 
\draw    (130.67,211) -- (130.67,311) ;
%Straight Lines [id:da7093753873233952] 
\draw    (30.67,311) -- (130.67,311) ;
%Straight Lines [id:da6253607222990134] 
\draw    (61,240) -- (130.67,311) ;
%Straight Lines [id:da044600977978810796] 
\draw    (130.67,311) -- (201,382.67) ;
%Straight Lines [id:da5206555343258983] 
\draw [fill={rgb, 255:red, 208; green, 2; blue, 27 }  ,fill opacity=1 ]   (130.67,311) -- (133.03,308.72) -- (202.33,242) ;
%Straight Lines [id:da4027768267541241] 
\draw    (130.67,311) -- (61,382) ;
%Flowchart: Connector [id:dp35266956319756404] 
\draw  [fill={rgb, 255:red, 0; green, 0; blue, 0 }  ,fill opacity=1 ] (126.67,211) .. controls (126.67,208.79) and (128.46,207) .. (130.67,207) .. controls (132.88,207) and (134.67,208.79) .. (134.67,211) .. controls (134.67,213.21) and (132.88,215) .. (130.67,215) .. controls (128.46,215) and (126.67,213.21) .. (126.67,211) -- cycle ;
%Flowchart: Connector [id:dp9365211913392373] 
\draw  [fill={rgb, 255:red, 0; green, 0; blue, 0 }  ,fill opacity=1 ] (57,240) .. controls (57,237.79) and (58.79,236) .. (61,236) .. controls (63.21,236) and (65,237.79) .. (65,240) .. controls (65,242.21) and (63.21,244) .. (61,244) .. controls (58.79,244) and (57,242.21) .. (57,240) -- cycle ;
%Flowchart: Connector [id:dp16054668103759195] 
\draw  [fill={rgb, 255:red, 0; green, 0; blue, 0 }  ,fill opacity=1 ] (26.67,311) .. controls (26.67,308.79) and (28.46,307) .. (30.67,307) .. controls (32.88,307) and (34.67,308.79) .. (34.67,311) .. controls (34.67,313.21) and (32.88,315) .. (30.67,315) .. controls (28.46,315) and (26.67,313.21) .. (26.67,311) -- cycle ;
%Flowchart: Connector [id:dp5932779690615457] 
\draw  [fill={rgb, 255:red, 0; green, 0; blue, 0 }  ,fill opacity=1 ] (57,382) .. controls (57,379.79) and (58.79,378) .. (61,378) .. controls (63.21,378) and (65,379.79) .. (65,382) .. controls (65,384.21) and (63.21,386) .. (61,386) .. controls (58.79,386) and (57,384.21) .. (57,382) -- cycle ;
%Flowchart: Connector [id:dp7243343626999181] 
\draw  [fill={rgb, 255:red, 0; green, 0; blue, 0 }  ,fill opacity=1 ] (126.67,411) .. controls (126.67,408.79) and (128.46,407) .. (130.67,407) .. controls (132.88,407) and (134.67,408.79) .. (134.67,411) .. controls (134.67,413.21) and (132.88,415) .. (130.67,415) .. controls (128.46,415) and (126.67,413.21) .. (126.67,411) -- cycle ;
%Flowchart: Connector [id:dp9495496589685393] 
\draw  [fill={rgb, 255:red, 0; green, 0; blue, 0 }  ,fill opacity=1 ] (197,382.67) .. controls (197,380.46) and (198.79,378.67) .. (201,378.67) .. controls (203.21,378.67) and (205,380.46) .. (205,382.67) .. controls (205,384.88) and (203.21,386.67) .. (201,386.67) .. controls (198.79,386.67) and (197,384.88) .. (197,382.67) -- cycle ;
%Flowchart: Connector [id:dp6172920262882335] 
\draw  [fill={rgb, 255:red, 0; green, 0; blue, 0 }  ,fill opacity=1 ] (226.67,311) .. controls (226.67,308.79) and (228.46,307) .. (230.67,307) .. controls (232.88,307) and (234.67,308.79) .. (234.67,311) .. controls (234.67,313.21) and (232.88,315) .. (230.67,315) .. controls (228.46,315) and (226.67,313.21) .. (226.67,311) -- cycle ;
%Flowchart: Connector [id:dp6484002810590195] 
\draw  [fill={rgb, 255:red, 0; green, 0; blue, 0 }  ,fill opacity=1 ] (198.33,242) .. controls (198.33,239.79) and (200.12,238) .. (202.33,238) .. controls (204.54,238) and (206.33,239.79) .. (206.33,242) .. controls (206.33,244.21) and (204.54,246) .. (202.33,246) .. controls (200.12,246) and (198.33,244.21) .. (198.33,242) -- cycle ;
%Straight Lines [id:da32365958597058464] 
\draw    (476.67,311) -- (576.67,311) ;
%Flowchart: Connector [id:dp4151632393882072] 
\draw  [fill={rgb, 255:red, 0; green, 0; blue, 0 }  ,fill opacity=1 ] (472.67,311) .. controls (472.67,308.79) and (474.46,307) .. (476.67,307) .. controls (478.88,307) and (480.67,308.79) .. (480.67,311) .. controls (480.67,313.21) and (478.88,315) .. (476.67,315) .. controls (474.46,315) and (472.67,313.21) .. (472.67,311) -- cycle ;
%Straight Lines [id:da8125690112790387] 
\draw [color={rgb, 255:red, 155; green, 155; blue, 155 }  ,draw opacity=1 ][line width=2.25]    (476.67,311) -- (476.67,411) ;
%Straight Lines [id:da7136011183655138] 
\draw    (476.67,211) -- (476.67,311) ;
%Straight Lines [id:da989461136937205] 
\draw [color={rgb, 255:red, 155; green, 155; blue, 155 }  ,draw opacity=1 ][line width=2.25]    (376.67,311) -- (476.67,311) ;
%Straight Lines [id:da09708277286253697] 
\draw    (407,240) -- (476.67,311) ;
%Straight Lines [id:da5180345033488976] 
\draw    (476.67,311) -- (547,382.67) ;
%Straight Lines [id:da1080097592808571] 
\draw    (476.67,311) -- (548.33,242) ;
%Straight Lines [id:da4177902889485787] 
\draw [color={rgb, 255:red, 155; green, 155; blue, 155 }  ,draw opacity=1 ][line width=2.25]    (476.67,311) -- (407,382) ;
%Flowchart: Connector [id:dp113327553644883] 
\draw  [fill={rgb, 255:red, 0; green, 0; blue, 0 }  ,fill opacity=1 ] (472.67,211) .. controls (472.67,208.79) and (474.46,207) .. (476.67,207) .. controls (478.88,207) and (480.67,208.79) .. (480.67,211) .. controls (480.67,213.21) and (478.88,215) .. (476.67,215) .. controls (474.46,215) and (472.67,213.21) .. (472.67,211) -- cycle ;
%Flowchart: Connector [id:dp9874417238161601] 
\draw  [fill={rgb, 255:red, 0; green, 0; blue, 0 }  ,fill opacity=1 ] (403,240) .. controls (403,237.79) and (404.79,236) .. (407,236) .. controls (409.21,236) and (411,237.79) .. (411,240) .. controls (411,242.21) and (409.21,244) .. (407,244) .. controls (404.79,244) and (403,242.21) .. (403,240) -- cycle ;
%Flowchart: Connector [id:dp9533943039607864] 
\draw  [fill={rgb, 255:red, 0; green, 0; blue, 0 }  ,fill opacity=1 ] (372.67,311) .. controls (372.67,308.79) and (374.46,307) .. (376.67,307) .. controls (378.88,307) and (380.67,308.79) .. (380.67,311) .. controls (380.67,313.21) and (378.88,315) .. (376.67,315) .. controls (374.46,315) and (372.67,313.21) .. (372.67,311) -- cycle ;
%Flowchart: Connector [id:dp1375020946908252] 
\draw  [fill={rgb, 255:red, 0; green, 0; blue, 0 }  ,fill opacity=1 ] (403,382) .. controls (403,379.79) and (404.79,378) .. (407,378) .. controls (409.21,378) and (411,379.79) .. (411,382) .. controls (411,384.21) and (409.21,386) .. (407,386) .. controls (404.79,386) and (403,384.21) .. (403,382) -- cycle ;
%Flowchart: Connector [id:dp9944974378267495] 
\draw  [fill={rgb, 255:red, 0; green, 0; blue, 0 }  ,fill opacity=1 ] (472.67,411) .. controls (472.67,408.79) and (474.46,407) .. (476.67,407) .. controls (478.88,407) and (480.67,408.79) .. (480.67,411) .. controls (480.67,413.21) and (478.88,415) .. (476.67,415) .. controls (474.46,415) and (472.67,413.21) .. (472.67,411) -- cycle ;
%Flowchart: Connector [id:dp9238277905890337] 
\draw  [fill={rgb, 255:red, 0; green, 0; blue, 0 }  ,fill opacity=1 ] (543,382.67) .. controls (543,380.46) and (544.79,378.67) .. (547,378.67) .. controls (549.21,378.67) and (551,380.46) .. (551,382.67) .. controls (551,384.88) and (549.21,386.67) .. (547,386.67) .. controls (544.79,386.67) and (543,384.88) .. (543,382.67) -- cycle ;
%Flowchart: Connector [id:dp3922906473848795] 
\draw  [fill={rgb, 255:red, 0; green, 0; blue, 0 }  ,fill opacity=1 ] (572.67,311) .. controls (572.67,308.79) and (574.46,307) .. (576.67,307) .. controls (578.88,307) and (580.67,308.79) .. (580.67,311) .. controls (580.67,313.21) and (578.88,315) .. (576.67,315) .. controls (574.46,315) and (572.67,313.21) .. (572.67,311) -- cycle ;
%Flowchart: Connector [id:dp09452135121702354] 
\draw  [fill={rgb, 255:red, 0; green, 0; blue, 0 }  ,fill opacity=1 ] (544.33,242) .. controls (544.33,239.79) and (546.12,238) .. (548.33,238) .. controls (550.54,238) and (552.33,239.79) .. (552.33,242) .. controls (552.33,244.21) and (550.54,246) .. (548.33,246) .. controls (546.12,246) and (544.33,244.21) .. (544.33,242) -- cycle ;
%Straight Lines [id:da8566027465672807] 
\draw [color={rgb, 255:red, 155; green, 155; blue, 155 }  ,draw opacity=1 ][line width=2.25]    (310.71,311) -- (376.67,311) ;
%Straight Lines [id:da19957923164317037] 
\draw    (407,240) -- (310.71,311) ;
%Straight Lines [id:da6045648393218841] 
\draw [color={rgb, 255:red, 155; green, 155; blue, 155 }  ,draw opacity=1 ][line width=2.25]    (310.71,311) -- (407,382) ;
%Straight Lines [id:da0593829932053076] 
\draw    (310.71,311) -- (476.67,411) ;
%Straight Lines [id:da7925087656311398] 
\draw    (310.71,311) -- (476.67,211) ;
%Straight Lines [id:da26698941064280857] 
\draw    (310.71,311) -- (548.33,242) ;
%Straight Lines [id:da6816848666532045] 
\draw    (310.71,311) -- (547,382.67) ;
%Flowchart: Connector [id:dp5567203113799104] 
\draw  [fill={rgb, 255:red, 208; green, 2; blue, 27 }  ,fill opacity=1 ] (126.67,311) .. controls (126.67,308.79) and (128.46,307) .. (130.67,307) .. controls (132.88,307) and (134.67,308.79) .. (134.67,311) .. controls (134.67,313.21) and (132.88,315) .. (130.67,315) .. controls (128.46,315) and (126.67,313.21) .. (126.67,311) -- cycle ;
%Flowchart: Connector [id:dp10202476090799029] 
\draw  [fill={rgb, 255:red, 208; green, 2; blue, 27 }  ,fill opacity=1 ] (310.71,311) .. controls (310.71,308.79) and (312.5,307) .. (314.71,307) .. controls (316.92,307) and (318.71,308.79) .. (318.71,311) .. controls (318.71,313.21) and (316.92,315) .. (314.71,315) .. controls (312.5,315) and (310.71,313.21) .. (310.71,311) -- cycle ;
%Flowchart: Connector [id:dp375924343548222] 
\draw  [fill={rgb, 255:red, 208; green, 2; blue, 27 }  ,fill opacity=1 ] (472.67,311) .. controls (472.67,308.79) and (474.46,307) .. (476.67,307) .. controls (478.88,307) and (480.67,308.79) .. (480.67,311) .. controls (480.67,313.21) and (478.88,315) .. (476.67,315) .. controls (474.46,315) and (472.67,313.21) .. (472.67,311) -- cycle ;
%Shape: Arc [id:dp3682263428765633] 
\draw  [draw opacity=0] (314.98,308.6) .. controls (318.27,241.23) and (373.46,186.78) .. (442.47,185.02) .. controls (514.76,183.17) and (574.85,239.86) .. (576.68,311.64) .. controls (576.72,313.17) and (576.73,314.71) .. (576.72,316.23) -- (445.78,314.97) -- cycle ; \draw   (314.98,308.6) .. controls (318.27,241.23) and (373.46,186.78) .. (442.47,185.02) .. controls (514.76,183.17) and (574.85,239.86) .. (576.68,311.64) .. controls (576.72,313.17) and (576.73,314.71) .. (576.72,316.23) ;  
%Shape: Arc [id:dp5556400676322113] 
\draw  [draw opacity=0][line width=2.25]  (475.22,411.22) .. controls (420.66,410.44) and (376.67,365.97) .. (376.67,311.23) .. controls (376.67,311.15) and (376.67,311.08) .. (376.67,311) -- (476.67,311.23) -- cycle ; \draw  [color={rgb, 255:red, 155; green, 155; blue, 155 }  ,draw opacity=1 ][line width=2.25]  (475.22,411.22) .. controls (420.66,410.44) and (376.67,365.97) .. (376.67,311.23) .. controls (376.67,311.15) and (376.67,311.08) .. (376.67,311) ;  
%Shape: Arc [id:dp7264704850231363] 
\draw  [draw opacity=0][line width=2.25]  (376.67,311) .. controls (376.66,310.75) and (376.66,310.5) .. (376.66,310.25) .. controls (376.66,255.02) and (421.44,210.25) .. (476.66,210.25) .. controls (531.89,210.25) and (576.66,255.02) .. (576.66,310.25) .. controls (576.66,365.48) and (531.89,410.25) .. (476.66,410.25) .. controls (476.47,410.25) and (476.27,410.25) .. (476.07,410.25) -- (476.66,310.25) -- cycle ; \draw  [line width=2.25]  (376.67,311) .. controls (376.66,310.75) and (376.66,310.5) .. (376.66,310.25) .. controls (376.66,255.02) and (421.44,210.25) .. (476.66,210.25) .. controls (531.89,210.25) and (576.66,255.02) .. (576.66,310.25) .. controls (576.66,365.48) and (531.89,410.25) .. (476.66,410.25) .. controls (476.47,410.25) and (476.27,410.25) .. (476.07,410.25) ;  

% Text Node
\draw (164,200.32) node [anchor=north west][inner sep=0.75pt]    {$1$};
% Text Node
\draw (227,255.32) node [anchor=north west][inner sep=0.75pt]    {$2$};
% Text Node
\draw (228,353.32) node [anchor=north west][inner sep=0.75pt]    {$3$};
% Text Node
\draw (167,407.32) node [anchor=north west][inner sep=0.75pt]    {$4$};
% Text Node
\draw (75,402.32) node [anchor=north west][inner sep=0.75pt]    {$5$};
% Text Node
\draw (23,351.32) node [anchor=north west][inner sep=0.75pt]    {$6$};
% Text Node
\draw (23,260.32) node [anchor=north west][inner sep=0.75pt]    {$7$};
% Text Node
\draw (83,200.32) node [anchor=north west][inner sep=0.75pt]    {$8$};
% Text Node
\draw (502,222.32) node [anchor=north west][inner sep=0.75pt]    {$17$};
% Text Node
\draw (571,256.32) node [anchor=north west][inner sep=0.75pt]    {$17$};
% Text Node
\draw (573,351.32) node [anchor=north west][inner sep=0.75pt]    {$18$};
% Text Node
\draw (514,405.32) node [anchor=north west][inner sep=0.75pt]    {$19$};
% Text Node
\draw (422,404.32) node [anchor=north west][inner sep=0.75pt]    {$19$};
% Text Node
\draw (386,339.32) node [anchor=north west][inner sep=0.75pt]    {$20$};
% Text Node
\draw (385,270.32) node [anchor=north west][inner sep=0.75pt]    {$21$};
% Text Node
\draw (430,198.32) node [anchor=north west][inner sep=0.75pt]    {$21$};
% Text Node
\draw (133,244.32) node [anchor=north west][inner sep=0.75pt]    {$9$};
% Text Node
\draw (170,272.32) node [anchor=north west][inner sep=0.75pt]    {$9$};
% Text Node
\draw (172.67,312.4) node [anchor=north west][inner sep=0.75pt]    {$10$};
% Text Node
\draw (147.33,345.07) node [anchor=north west][inner sep=0.75pt]    {$11$};
% Text Node
\draw (109.67,352.4) node [anchor=north west][inner sep=0.75pt]    {$11$};
% Text Node
\draw (74.33,332.07) node [anchor=north west][inner sep=0.75pt]    {$12$};
% Text Node
\draw (64.33,292.07) node [anchor=north west][inner sep=0.75pt]    {$13$};
% Text Node
\draw (91.67,257.4) node [anchor=north west][inner sep=0.75pt]    {$13$};
% Text Node
\draw (477,222.32) node [anchor=north west][inner sep=0.75pt]    {$1$};
% Text Node
\draw (513.67,274.32) node [anchor=north west][inner sep=0.75pt]    {$2$};
% Text Node
\draw (531,310.99) node [anchor=north west][inner sep=0.75pt]    {$3$};
% Text Node
\draw (482.33,324.32) node [anchor=north west][inner sep=0.75pt]    {$4$};
% Text Node
\draw (465.67,320.99) node [anchor=north west][inner sep=0.75pt]    {$5$};
% Text Node
\draw (450,313.32) node [anchor=north west][inner sep=0.75pt]    {$6$};
% Text Node
\draw (445.33,294.65) node [anchor=north west][inner sep=0.75pt]    {$7$};
% Text Node
\draw (461.33,282.32) node [anchor=north west][inner sep=0.75pt]    {$8$};
% Text Node
\draw (440,229.65) node [anchor=north west][inner sep=0.75pt]    {$9$};
% Text Node
\draw (449.33,250.32) node [anchor=north west][inner sep=0.75pt]    {$10$};
% Text Node
\draw (378.67,175.65) node [anchor=north west][inner sep=0.75pt]    {$11$};
% Text Node
\draw (492,368.32) node [anchor=north west][inner sep=0.75pt]    {$12$};
% Text Node
\draw (441.67,374.73) node [anchor=north west][inner sep=0.75pt]    {$13$};
% Text Node
\draw (338.33,345.07) node [anchor=north west][inner sep=0.75pt]    {$14$};
% Text Node
\draw (347.67,296.4) node [anchor=north west][inner sep=0.75pt]    {$15$};
% Text Node
\draw (356.33,247.4) node [anchor=north west][inner sep=0.75pt]    {$16$};
% Text Node
\draw (120.67,439.07) node [anchor=north west][inner sep=0.75pt]    {$( a)$};
% Text Node
\draw (466.67,439.1) node [anchor=north west][inner sep=0.75pt]    {$( b)$};

\end{tikzpicture}
   
    \caption{(a) a 13-edge-coloring of $W_8$ without any rainbow $F_5$; (b) a 21-edge-coloring of $W_8(2)$ without any rainbow $F_6$, but with a rainbow $F_5$ (highlighted in gray).}
    \label{fig:extremal}
\end{figure}

\section{Conclusions and remarks} \label{sec:final}

Combining the results of Theorems~\ref{thm:W-F}, \ref{thm:Ws-F-and-Wd-theta}, \ref{thm:extremal-1}, and \ref{thm:extremal-2}, we establish Theorem~\ref{thm:main-1}. Furthermore, the synthesis of Theorems~\ref{thm:Ws-F-and-Wd-theta} and \ref{thm:extremal-2} yields Theorem~\ref{thm:main-2}. As promised, we demonstrate that the conditions \(t \geq 6\) and \(t \geq 7\) specified in the second and third assertions of Theorem~\ref{thm:main-1} are critical, as evidenced by the following theorem, which is of independent interest.

\begin{thm}\label{thm5}
For integer $d\geq 5 $, we have
$\mathrm{rb}(W_d(2) ,F_5)\leq \left\lfloor \frac{33}{14}d \right\rfloor + 1$ and $\mathrm{rb}(W_d(3) ,F_6)\leq \left\lfloor \frac{55}{16}d \right\rfloor + 1$.
\end{thm}

\begin{proof}
Let \( m \in \{5, 6\} \), \( k_5 = \left\lfloor \frac{33}{14}d \right\rfloor + 1 \), and \( k_6 = \left\lfloor \frac{55}{16}d \right\rfloor + 1 \). Define \( W_{d}(m-3) \) as an edge-colored graph with \( k_m \) colors. Suppose, for the sake of contradiction, that no rainbow $ F_m $ subgraph exists in $ W_d(m-3) $.
Thus, each $F_m$ contained in $W_{d}(m-3)$ has a color $c$ with $p^{(m)}(c):=p(c,W_{d}(m-3),F_m)\ge 1$.
Set $A^{(m)}_i:=A_i(W_{d}(m-3))$ and $p^{(m)}_j(c):=p_j(c,W_{d}(m-3),F_m)$.

In the graph $W_{d}(m-2)$, each boundary vertex $w$ is in $2m-6$ $F_5$-subgraphs centered at $w$. Thus, the number of $F_m$ contained in $W_d(m-3)$ is exactly $(3m-9)d$, and moreover, 
each spoke appears in $\alpha_m$ distinct $F_m$-subgraphs and each rim edge in exactly $ \beta_m $, where $\alpha_5=\beta_5=14$, $\alpha_6=19$ and $\beta_6=24$,
For a color $c\in A$, let $x:=x(c)$ and $y:=y(c)$ be the number of spokes and rim edges colored $c$. Thus, for any $c\in A_i$ we have
\begin{align} \label{eq-51}
    2p^{(m)}(c)\le 2p^{(m)}(c)+p^{(m)}_1(c)\le \sum\limits_{j\ge 1}jp^{(m)}_j(c) = \alpha_m x+\beta_my\leq \beta_m (x+y)=\beta_m i.
\end{align}
By employing the method analogous to those used in establishing Lemmas~\ref{lem:p2} and~\ref{lem:p3}, we derive
\begin{equation} \label{eq-50}
    \begin{cases}
p^{(5)}(c) \leq 9  & \text{for any } c \in A^{(5)}_2, \\
p^{(5)}(c) \leq 15 & \text{for any } c \in A^{(5)}_3, \\
p^{(6)}(c) \leq 15 & \text{for any } c \in A^{(6)}_2, \\
p^{(6)}(c) \leq 27 & \text{for any } c \in A^{(6)}_3,
\end{cases}
\end{equation}
where the bounds arise from four distinct configurations:
\begin{itemize}
    \item when two consecutive spokes are colored $c$, there exist precisely nine $F_5$-subgraphs containing these spokes in $W_d(2)$;
    \item when three consecutive spokes are colored $c$, at most fifteen $F_5$-subgraphs include at least two of these spokes in $W_d(2)$;
    \item when two consecutive rim edges are colored $c$, there exist at most fifteen $F_6$-subgraphs containing these rim edges in $W_d(3)$;
    \item when three consecutive rim edges are colored $c$, at most twenty-seven $F_6$-subgraphs include at least two of these rim edges in $W_d(3)$.
\end{itemize}
We require \( d \geq 5 \); otherwise, the first two bounds may become \( 10 \) and \( 16 \), disturbing the afterward calculations. Although assuming \( d \geq 6 \) could reduce the last bound to \( 24 \), this adjustment does not strengthen our final conclusion. Besides, we have
\begin{align} \label{eq-52}
    (m-2)d=|E(W_{d}(m-3))|=\sum\limits_{i\ge 1}i|A^{(m)}_i|.
\end{align}

With these bounds in hand, we deduce
\begin{align}
    \notag 6d &\le \sum\limits_{c\in [k] }p^{(5)}(c) = \sum\limits_{i\ge 2}\sum\limits_{c\in A^{(5)}_i}p^{(5)}(c) = \sum\limits_{c\in A^{(5)}_2}p^{(5)}(c)+\sum\limits_{c\in A^{(5)}_3}p^{(5)}(c)+\sum\limits_{i\ge 4}\sum\limits_{c\in A^{(5)}_i}p^{(5)}(c)
    \\
   \notag  &\overset{\eqref{eq-51}\eqref{eq-50}}{\le} 9|A^{(5)}_2|+15|A^{(5)}_3|+\sum\limits_{i\ge 4} 7i|A^{(5)}_i|\\
   \notag &\overset{\eqref{eq-52}}{=} 9|A^{(5)}_2|+15|A^{(5)}_3|+7(3d-|A^{(5)}_1|-2|A^{(5)}_2|-3|A^{(5)}_3|)=21d-(7|A^{(5)}_1|+5|A^{(5)}_2|+6|A^{(5)}_3|).
\end{align}
It follows 
\begin{align*}
    \frac{3}{4}|A^{(5)}_1|+\frac{2}{4}|A^{(5)}_2|+\frac{1}{4}|A^{(5)}_3| \le \frac{3}{28} \bigg(7|A^{(5)}_1|+5|A^{(5)}_2|+6|A^{(5)}_3|\bigg)\le \frac{45}{28}d.
\end{align*}
Thus,
\begin{align*}
    k_5=\bigg\lfloor \frac{33}{14}d  \bigg\rfloor +1 &=\sum_{i\geq 1}|A^{(5)}_i| 
    \le \frac{3}{4}|A^{(5)}_1|+\frac{2}{4}|A^{(5)}_2|+\frac{1}{4}|A^{(5)}_3|+\frac{1}{4}\sum_{i\geq 1}i|A^{(5)}_i|\\
    &\overset{\eqref{eq-52}}{=}\frac{3}{4}|A^{(5)}_1|+\frac{2}{4}|A^{(5)}_2|+\frac{1}{4}|A^{(5)}_3|+\frac{3}{4}d
    \le\frac{45}{28}d+\frac{3}{4}d=\frac{33}{14}d.
\end{align*}
For $m=6$, we have 
\begin{align}
    \notag 9d &\le \sum\limits_{c\in [k] }p^{(6)}(c) = \sum\limits_{i\ge 2}\sum\limits_{c\in A^{(6)}_i}p^{(6)}(c) = \sum\limits_{c\in A^{(6)}_2}p^{(6)}(c)+\sum\limits_{c\in A^{(6)}_3}p^{(6)}(c)+\sum\limits_{i\ge 4}\sum\limits_{c\in A^{(6)}_i}p^{(6)}(c)
    \\
   \notag  &\overset{\eqref{eq-51}\eqref{eq-50}}{\le} 15|A^{(6)}_2|+27|A^{(6)}_3|+\sum\limits_{i\ge 4} 12i|A^{(6)}_i| \\
   \notag &\overset{\eqref{eq-52}}{=} 15|A^{(6)}_2|+27|A^{(6)}_3|+12(4d-|A^{(6)}_1|-2|A^{(6)}_2|-3|A^{(6)}_3|)=48d-(12|A^{(6)}_1|+9|A^{(6)}_2|+9|A^{(6)}_3|).
\end{align}
It follows 
\begin{align*}
    \frac{3}{4}|A^{(6)}_1|+\frac{2}{4}|A^{(6)}_2|+\frac{1}{4}|A^{(6)}_3| \le \frac{3}{16} \bigg(4|A^{(6)}_1|+3|A^{(6)}_2|+3|A^{(6)}_3|\bigg)\le \frac{39}{16}d.
\end{align*}
Thus,
\begin{align*}
    k_6=\bigg\lfloor \frac{55}{16}d  \bigg\rfloor +1 &=\sum_{i\geq 1}|A^{(6)}_i| 
    \le \frac{3}{4}|A^{(6)}_1|+\frac{2}{4}|A^{(6)}_2|+\frac{1}{4}|A^{(6)}_3|+\frac{1}{4}\sum_{i\geq 1}i|A^{(6)}_i|\\
    &\overset{\eqref{eq-52}}{=}\frac{3}{4}|A^{(6)}_1|+\frac{2}{4}|A^{(6)}_2|+\frac{1}{4}|A^{(6)}_3|+ d
    \le\frac{39}{16}d+ d=\frac{55}{16}d.
\end{align*}
These two contradictions complete the proof.
\end{proof}

We do not know whether the two bounds in Theorem~\ref{thm5} are tight. Thus, determining the exact values of 
$\mathrm{rb}(W_d(2), F_5)$ and $\mathrm{rb}(W_d(3), F_6)$ remains an interesting open problem. More broadly, the problem of fixing $\mathrm{rb}(W_d(s), F_t)$ is still open for $s = 2$ and $t \leq 5$, or $s \geq 3$ and $t \leq 6$.

\section*{Acknowledgments}

The authors thank Ping Chen, Yali Wu, and Dezhi Zou for participating in the initial discussion of this paper's topic.

\bibliography{ref}
\bibliographystyle{acm}

\end{document}